\documentclass{amsart}

\usepackage[latin1]{inputenc}
\usepackage[english]{babel}
\usepackage{amstext}
\usepackage{latexsym}
\usepackage{amsmath}
\usepackage{amssymb}
\usepackage{amsfonts}
\usepackage{graphicx}
\usepackage{amssymb}
\usepackage{amsthm}
\usepackage{epsfig}
\usepackage{tikz}
\usetikzlibrary{matrix,arrows,decorations.pathmorphing}
\usepackage{graphicx}
\usepackage{enumitem}

\newtheorem{defn}{Definition}

\newtheorem{lemma}{Lemma}

\newtheorem{prop}{Proposition}
\newtheorem{theo}{Theorem}
\newtheorem{example}{Example}
\newtheorem{claim}{Claim}
\newtheorem{theoA}{Theorem A}

\newtheorem{theoB}{Theorem B} 

\newtheorem{conj}{Conjecture}

\newcommand{\R}{\mathbb R}
\newcommand{\Z}{\mathbb Z}
\newcommand{\xn}{(x_n)_{n\in\mathbb{Z}}}
\newcommand{\N}{\mathbb{N}}
\newcommand{\Sig}{\hat{\Sigma}}

\begin{document}

\title[Continuity of Lyapunov exponents]{Continuity of Lyapunov exponents for non-uniformly fiber-bunched cocycles}

\author{Catalina Freijo}
\address{Departamento de Matem\'atica, Universidade Federal de Minas Gerais, Av. Ant\^onio Carlos 6627, 31270-901 Belo Horizonte
Minas Gerais, Brazil}
\email{catalinafreijo@gmail.com}

\author{Karina Marin}
\address{Departamento de Matem\'atica, Universidade Federal de Minas Gerais, Av. Ant\^onio Carlos 6627, 31270-901 Belo Horizonte
Minas Gerais, Brazil}
\email{kmarin@mat.ufmg.br}

\begin{abstract}
We provide conditions that imply the continuity of the Lyapunov exponents for non-uniformly fiber-bunched cocycles in $SL(2,\mathbb{R})$. The main theorem is an extension of the result of Backes, Brown and Butler and gives a partial answer to a conjecture of Marcelo Viana. 
\end{abstract}



\maketitle

\section{Introduction}

The notion of Lyapunov exponents goes back to the work of A. M. Lyapunov in the late 19th century about the stability theory for differential equations. It was extended to the field of ergodic theory by the results of Fustenberg-Kesten \cite{FK} and O\-se\-le\-dets \cite{O}. Lyapunov exponents also appear naturally in smooth dynamics through the concept of non-uniform hyperbolicity introduced by Pesin \cite{P}. 

The theory of Lyapunov exponents for linear cocycles grew into a very broad area and active field. In this work, we are concerned with the continuity of the Lyapunov exponents for linear cocycles in $SL(2,\mathbb{R})$. That is, we study how the Lyapunov exponents vary as functions of the cocycle.

Discontinuity of Lyapunov exponents is typical for continuous $SL(2,\mathbb{R})$-valued cocycles over an invertible base. This has been proved in Theorem C of \cite{B1} as a particular case of Ma\~n\'e-Bochi's Theorem. More precisely, in \cite{B1} was shown that the only $C^0$-continuity points of the Lyapunov exponents are the cocycles which are either uniformly hyperbolic or those with zero Lyapunov exponents.

Even though, dis\-con\-ti\-nu\-i\-ty is a common feature, there are some contexts where continuity has been established. Bocker and Viana \cite{BV} and Malheiro and Viana \cite{MV} proved continuity of Lyapunov exponents for random products of 2-dimensional matrices in the Bernoulli and in the Markov setting. In higher dimension, continuity of the Lyapunov exponents for i.i.d. random products of matrices has been announced by Avila, Eskin and Viana \cite{AEV}.

Still for 2-dimensional cocycles, Bocker and Viana \cite{BV} constructed an example of a locally constant cocycle with non-zero Lyapunov exponents that can be approximated in the H\"older topology by linear cocycles with zero Lyapunov exponents. Then, we cannot expect to have continuity of the Lyapunov exponents even if we consider higher regularity. Another counter-example in this setting has been constructed in \cite{Bu}.

A few years ago, Backes, Brown and Butler \cite{BBB} proved that the continuity of Lyapunov exponents holds when restricted to the realm of fiber-bunched H\"older cocycles over any hyperbolic system and for any ergodic probability measure with local product structure. The main feature that fiber-bunched cocycles exhibit is the existence of uniform invariant holonomies. In fact, the main theorem in \cite{BBB} establishes that for continuous cocycles that admit uniform stable and uniform unstable holonomies, denoted by $H^s$ and $H^u$ respectively, we have the following: $$\text{If }\, (\hat{A}_k, H^{s,k}, H^{u,k})\stackrel{C^0}{\longrightarrow}(\hat{A},H^s,H^u), \text{ then } \lambda_{+}(\hat{A}_k)\to \lambda_{+}(\hat{A}).$$  In particular, their theorem extends \cite{BV} and \cite{MV}. 

More recently, Viana and Yang \cite{VY} were able to prove the continuity of Lyapunov exponents in the $C^0$ topology for a subset of linear cocycles when the transformation in the base is a uniformly expanding map. In their statement, the cocycle has non-zero Lyapunov exponents and it is not uniformly hyperbolic. This means that the Ma\~n\'e-Bochi phenomenon cannot be generalized to the non-invertible setting. The main observation is that given a uniformly expanding map, we can consider its natural extension and the lift of the cocycle. This new cocycle always admits a uniform stable holonomy. Therefore, the theorems in \cite{VY} suggest that the hypotheses in \cite{BBB} can be relaxed: we may only need to ask for the existence of a single uniform holonomy. 

\begin{conj} [Conjecture 6.3, \cite{SV}] 
\begin{equation*}
\begin{aligned}
\text{If }\, (\hat{A}_k, H^{u,k})\stackrel{C^0}{\longrightarrow}&(\hat{A},H^u)\, \text{ or }\, (\hat{A}_k, H^{s,k})\stackrel{C^0}{\longrightarrow}(\hat{A},H^s),\\
&\text{ then } \lambda_{+}(\hat{A}_k)\to \lambda_{+}(\hat{A}).
\end{aligned}
\end{equation*}
\end{conj}

The results in the present work give a partial answer to this conjecture. More precisely, we prove that the conjecture is true if the cocycle $\hat{A}$ is H\"older continuous and non-uniformly fiber-bunched. This notion was introduced in \cite{V} and implies the existence of holonomies with weaker properties than the uniform ones. 

In the following theorem the map in the base is a hyperbolic homeomorphism and $\hat{\mu}$ is an ergodic $\hat{f}$-invariant probability measure with local product structure and fully supported. We refer the reader to next section for the precise definitions.

\begin{theoA} Let $\hat{A}$ be a H\"older $SL(2,\mathbb{R})$-valued linear cocycle such that $\hat{A}$ is non-uniformly fiber-bunched and admits a uniform stable holonomy. Consider a sequence $(\hat{A}_k, H^{s,k})$ such that $\hat{A}_k\to \hat{A}$ in the H\"older topology and $H^{s,k}\to H^s$ in the $C^0$ topology. Then, $\lambda_{+}(\hat{A}_k)\to \lambda_{+}(\hat{A})$. 
\end{theoA}

It is possible to deduce an analogous result for cocycles admitting only a uniform unstable holonomy applying Theorem A to $(\hat{f}^{-1}, \hat{A}^{-1})$. 

In particular, the theorem above shows that the Lyapunov exponents of non-uniformly fiber-bunched H\"older cocycles over a uniformly expanding map vary continuously with the cocycle. 

Although, the non-uniform fiber-bunching condition in the hypotheses of Theorem A implies the existence of some kind of unstable holonomy for $\hat{A}$, this new type of holonomy does not share the uniform properties needed to apply the argument in \cite{BBB}. Several results need to be extended to our context in order to conclude the theorem. 
 
Observe that compared with the conjecture, Theorem A asks for higher regularity in the cocycle. This is the case, because in the $C^0$ topology, non-uniform fiber-bunching does not implies the existence of non-uniform holonomies. Moreover, if the sequence $\{\hat{A}_k\}$ converges to $\hat{A}$ in the H\"older topology, then the non-uniform holonomies exist for every $\hat{A}_k$ and they are continuous with the cocycle. This is a key property in the proof of the theorem. More precisely, it is enough to have that $\{\hat{A}_k\}$ converges to $\hat{A}$ in the $C^0$ topology and there exists $C>0$ such that the H\"older constant of $\hat{A}_k$, $H(\hat{A}_k)$, verifies $H(\hat{A}_k)<C$ for every $k\in \mathbb{N}$. 

We remark that the existence of a uniform stable holonomy in Theorem A cannot be removed. In fact, the example of dis\-con\-ti\-nu\-i\-ty in \cite{BV} can be taken to be non-uniformly fiber-bunched. Therefore, we cannot expect continuity of the Lyapunov exponents to hold in the space of non-uniformly fiber-bunched cocycles without some extra hypotheses. 

In the next result, which is a consequence of the tools developed to prove Theorem A, we analyze the particular case of locally constant cocycles. Let $\hat{f}$ be the left-shift map and $\hat{\mu}$ be a fully supported Bernoulli measure. 

\begin{theoB} Let $\hat{A}$ be a H\"older $SL(2,\mathbb{R})$-valued linear cocycle such that $\hat{A}$ is non-uniformly fiber-bunched, locally constant and irreducible. If $\hat{A}_k\to \hat{A}$ in the H\"older topology, then $\lambda_{+}(\hat{A}_k)\to \lambda_{+}(\hat{A})$. 
\end{theoB} 

We refer the reader to the next section for the precise definitions, but we remark that irreducibility is a $C^{\alpha}$ dense condition among locally constant cocycles, see for example Theorem 7.12 of \cite{V}. This shows that the example of discontinuity in \cite{BV} is not typical among non-uniformly fiber-bunched cocycles. 

\section{Preliminaries and statements}

In this section we provide the necessary definitions to give the precise statements of the theorems. Without loss of generality, we consider the map in the base as being a sub-shift of finite type.

Let $Q=(q_{i,j})_{1\leq i,j \leq d}$ be a matrix with $q_{i,j}\in \{ 0,1\}$. The sub-shift of finite type $\hat{\Sigma}$ associated to the matrix $Q$ is the subset of the bi-infinite sequences $\{1,... ,d \}^{\Z}$ satisfying $$\Sig=\{\xn \colon q_{x_nx_{n+1}}=1 \text{ for every } n\in\Z\}.$$ We require that each row and column of $Q$ contains at least one non-zero entry. 

For any $\rho\in(0,1)$, we define a metric $d_\rho$ in $\Sig$ by $$d_\rho(\hat{x},\hat{y})=\rho^{N(\hat{x},\hat{y})},$$ where $N(\hat{x},\hat{y})=\max\{N\geq 0; x_n=y_n \text{ for every } |n|<N\}$. Since the topologies given by the different distances $d_{\rho}$ are equivalent, from now on we consider $\rho$ fixed.

Let $\hat{f}\colon\Sig\to\Sig$ be the left-shift map defined by $\hat{f}\xn=(x_{n+1})_{n\in\Z}$. The map $\hat{f}$ is a hyperbolic homeomorphism such that for every $\xn\in\Sig$ the local stable and unstable sets are given by
\begin{equation*}
\begin{aligned}
W^s_{loc}(\hat{x})=&\{(y_n)_{n\in\Z}\in\Sig: \ y_n=x_n \text{ with } n\geq 0\},\\ 
W^u_{loc}(\hat{x})=&\{(y_n)_{n\in\Z}\in\Sig: \ y_n=x_n \text{ with } n\leq 0\}.
\end{aligned}
\end{equation*}

Let $\sigma=1/\rho$ where $\rho$ is the constant in the definition of the distance. Observe that $\sigma$ is the expansion rate of $\hat{f}$. 

It is possible to express $\Sig$ locally as a product space if we consider the positive and the negative coordinates separately. Define
\begin{equation*}
\begin{aligned}
\Sigma^u=&\{(x_n)_{n\geq0}: q_{x_nx_{n+1}}=1 \text{ for every } n\geq 0\}\\
\Sigma^s=&\{(x_n)_{n\leq0}: q_{x_nx_{n+1}}=1 \text{ for every } n\leq -1\}.
\end{aligned}
\end{equation*}

We denote by $P^{u}\colon\Sig\to\Sigma^{u}$ and $P^{s}\colon\Sig\to\Sigma^{s}$ the projections obtained by dropping the negative and the positive coordinates, respectively, of a sequence in $\Sig$. 

For each $i\in\{1, ..., d\}$, define $[0;i]=\{\hat{x}\in\hat{\Sigma} : x_0=i\}$.

\begin{defn} An $\hat{f}$-invariant measure $\hat{\mu}$ has local product structure if there exists a continuous function $\psi\colon \Sig\to(0,\infty)$ such that for each $i\in \{1,...,d\}$, $$d\hat{\mu}\vert_{[0;i]}=\psi\cdot\left( d\mu^s\vert_{P^s([0;i])}\times d\mu^u\vert_{P^u([0;i])}\right),$$ where $\mu^s=P^s_{*}\hat{\mu}$ and $\mu^u=P^u_{*}\hat{\mu}$.
\end{defn}

For every $x\in \Sigma^u$, define $W^s_{loc}(x)=(P^u)^{-1}(x)$. Then, by Rokhlin \cite{R}, there exists a disintegration of $\hat{\mu}$ into conditional probabilities $\{\hat{\mu}_{x}\}_{x\in\Sigma^u}$ such that each $\hat{\mu}_x$ is supported on $W^s_{loc}(x).$ Observe that $d\hat{\mu}_x=\psi d\mu^s$. 

Given $x,y\in\Sigma^u\cap[0;i]$, the \emph{unstable holonomy map} $$h_{x,y}\colon W_{loc}^s(x)\to W^s_{loc}(y),$$ is defined by assigning to each $\hat{x}\in W^s_{loc}(x)$ the unique element $\hat{y}=h_{x,y}(\hat{x})\in W^s_{loc}(y)$ such that $\hat{y}\in W^u_{loc}(\hat{x})$. 

The lemma below is a well known consequence of the local product structure of the measure, see for instance \cite{BV04}. 

\begin{lemma}\label{lpe} Assume $\hat{\mu}$ has local product structure. Then, the measure $\hat{\mu}$ has a disintegration into conditional measures $\{\hat{\mu}_x\}_{x\in\Sigma^u}$ that vary continuously with $x$ in the weak-$^*$ topology. In fact, for every $x,y\in\Sigma^u$ in the same cylinder $[0;i]$, $$h_{x,y}:(W_{loc}^s(x),\hat{\mu}_x)\to(W^s_{loc}(y),\hat{\mu}_y)$$ is absolutely continuous, with Jacobian $R_{x,y}$ depending continuously on $(x,y)$.\end{lemma}

\subsection{Linear Cocycles.}

Let $\hat{A}\colon \Sig \to \mathit{SL}(2,\R)$ be a continuous map. The linear cocycle defined by $\hat{A}$ is the skew-product over $\hat{f}$, $F_{\hat{A}}\colon \Sig\times \mathbb{R}^2\to \Sig\times \mathbb{R}^2$, where $$F_{\hat{A}}(\hat{x},v)=(\hat{f}(\hat{x}), \hat{A}(\hat{x})v).$$ Since the base is fixed, from now on, we identify $\hat{A}$ with $F_{\hat{A}}$ and refer to $\hat{A}$ as a linear cocycle itself. 

For $n\geq 0$, let $$\hat{A}^n(\hat{x})=\hat{A}(\hat{f}^{n-1}(\hat{x}))\ldots \hat{A}(\hat{f}(\hat{x}))\hat{A}(\hat{x}),$$ and define $$\|\hat{A}^n(\hat{x})\|=\sup_{\scriptstyle\|v\|_{2}=1}\|\hat{A}^n(\hat{x})v\|_{\scriptscriptstyle2},$$ where $\|\cdot\|_{\scriptscriptstyle2}$ denotes the usual norm in $\R^2$.

By Furstenberg-Kesten \cite{FK}, for a continuous map $\hat{A}:\Sig\to SL(2, \R)$ and any $\hat{f}$-invariant probability measure $\hat{\mu}$,  $$\lambda_+(\hat{A},\hat{x})=\lim_{n\to\infty} \frac{1}{n}\log \|\hat{A}^n(\hat{x})\|$$ and $$\lambda_-(\hat{A},\hat{x})=\lim_{n\to\infty} \frac{1}{n} \log \|(\hat{A}^n(\hat{x}))^{-1}\|^{-1},$$ are well defined $\hat{\mu}$-almost every $\hat{x}\in \Sig$. 

Both $\lambda_{+}(\hat{A},\hat{x})$ and $\lambda_{-}(\hat{A},\hat{x})$ are called extremal Lyapunov exponents of $\hat{A}$. They are $\hat{f}$-invariant maps, thus when $\hat{\mu}$ is ergodic they are constant $\hat{\mu}$-almost everywhere. In that case, we denote them as $\lambda_+(\hat{A})$ and $\lambda_-(\hat{A})$. Since the cocycle takes values in $\mathit{SL}(2,\R)$, we conclude $\lambda_+(\hat{A})+\lambda_{-}(\hat{A})=0$. 

Moreover, for a fixed ergodic measure $\hat{\mu}$, Lemma 9.1 of \cite{LLE} states that the maps $\hat{A}\mapsto\lambda_+(\hat{A})$ and $\hat{A}\mapsto\lambda_-(\hat{A})$ are upper and lower semi-continuous, respectively, in the $C^0$ topology. 

Define the set $\Omega^s=\{(\hat{x},\hat{y})\in\Sig\times\Sig\colon\hat{y}\in W^s_{loc}(\hat{x})\}$.

\begin{defn} \label{holonomiafuerte}
A \emph{uniform stable holonomy} for $\hat{A}$ over $\hat{f}$ is a collection of linear isomorphisms $H^{s}_{\hat{x},\hat{y}}\colon \mathbb{R}^2\to \mathbb{R}^2 $, defined for every $\hat{x},\hat{y}$ in the same local stable set, which satisfy the following properties,
\vspace{0.2cm}

\begin{enumerate}[label=\emph{(\alph*)}]
\item $H^{s}_{\hat{y},\hat{z}}\circ H^{s}_{\hat{x},\hat{y}}= H^{s}_{\hat{x},\hat{z}}$ and $H^{s}_{\hat{x},\hat{x}}=Id$;
\vspace{0.2cm}

\item $H^{s}_{\hat{f}(\hat{x}),\hat{f}(\hat{y})}=\hat{A}({\hat{y}})\circ H^{s}_{\hat{x},\hat{y}}\circ \hat{A}(\hat{x})^{-1}$; 
\vspace{0.2cm}

\item $(\hat{x},\hat{y})\mapsto H^{s}_{\hat{x},\hat{y}}$ is continuous for every $(\hat{x},\hat{y})\in\Omega^s$.
\end{enumerate}
\vspace{0.2cm}

A \emph{uniform unstable holonomy} for $\hat{A}$ is defined analogously for points in the same local unstable set. We use the expression \emph{uniform invariant holonomies} to refer to both uniform stable and uniform unstable holonomies.
\end{defn}

Fixed a distance $d_{\rho}$ in $\Sig$, the set of $\alpha$-H\"older maps $\hat{A}\colon \Sig\to SL(2,\mathbb{R})$ is denoted by $\mathcal{S}_{\alpha}(\Sig,2)$. We equip this space with the $\alpha$-H\"older topology given by the distance $$D_{\alpha}(\hat{A},\hat{B})=\sup_{\hat{x}\in\hat{\Sigma}}\|\hat{A}(\hat{x})-\hat{B}(\hat{x})\| + H_{\alpha}(\hat{A}-\hat{B}),$$ where $H_{\alpha}(\hat{A})$ is the smallest constant $C>0$ such that $$\|\hat{A}(\hat{x})-\hat{A}(\hat{y})\|\leq C\; d_\rho(\hat{x},\hat{y})^{\alpha} \text{ for any } \hat{x},\hat{y}\in \hat{\Sigma} \text{ with } d_\rho(\hat{x},\hat{y})\leq 1.$$

We say that $\hat{A}\in \mathcal{S}_{\alpha}(\Sig,2)$ is \emph{$\alpha$-fiber-bunched} if there exists an $N>0$ such that for every $\hat{x}\in \Sig$, $$\|\hat{A}^N(\hat{x})\|\|(\hat{A}^N(\hat{x}))^{-1}\|^{-1}\rho^{\alpha N }<1.$$ Here $\rho$ is the constant in the definition of the distance $d_{\rho}$. 

The main property of $\alpha$-fiber-bunched cocycles is that they admit uniform invariant holonomies. For a proof of this fact we refer the reader to \cite{BGV}.

The following definition generalizes the notion of fiber-bunched mentioned above. However, it still allow us to prove the existence of invariant holonomies in a non-uniform sense. See Section 3.

\begin{defn}\label{nufb} Let $\hat{\mu}$ be an ergodic probability measure of $\hat{f}$. We say that $\hat{A}\in \mathcal{S}_{\alpha}(\Sig,2)$ is \emph{non-uniformly fiber-bunched} if the extremal Lyapunov exponents of $\hat{A}$ satisfy $$\lambda_{+}(\hat{A})-\lambda_{-}(\hat{A})=2\lambda_+(\hat{A})< \alpha \log \sigma,$$ or, equivalently $$\limsup_n\|\hat{A}^n(\hat{x})\|\|{\hat{A}^n(\hat{x})}^{-1}\|\; \rho^{\alpha n}<1 \text{ for } \hat{\mu}\text{-almost every } \hat{x},$$ where $\sigma=\frac{1}{\rho}$ is the expansion rate of $\hat{f}$.
\end{defn}

We observe that when we consider less regular cocycles, the upper bound for the top Lyapunov exponent $\lambda_{+}(\hat{A})$ decreases. Moreover, non-uniform fiber-bunching is a $C^0$-open condition. 

As the uniform invariant holonomies may not be unique (see, for example, Corollary 4.9 of \cite{KS}), we consider the cocycle and one of its holonomies in pairs. More precisely, $\mathcal{H}^s_{\alpha}$ is the set of pairs $(\hat{A}, H^{s})$ where $\hat{A}\in \mathcal{S}_{\alpha}(\Sig,2)$ and $H^{s}$ is a uniform stable holonomy for $\hat{A}$. 

Consider $\mathcal{H}^s_{\alpha}$ with the topology given by the inclusion
\begin{equation}\label{topo}
\mathcal{H}^s_{\alpha}\hookrightarrow \mathcal{S}_{\alpha}(\Sig,2)\times C^0(\Omega^s,SL(2,\R)).
\end{equation}
This means that a sequence $\{(\hat{A}_k,H^{s,k})\}_{k\in\N}$ converges to $(\hat{A},H^{s})$ in $\mathcal{H}^s_{\alpha}$ if $\hat{A}_k\to \hat{A}$ in the $\alpha$-H\"older topology and the uniform stable holonomy converges uniformly in every local stable set. 

\subsection{Statement of the theorems.}

With the previous definitions it is now possible to give the precise statements of the theorems. 

For the following result we consider $(\hat{f},\hat{\mu})$ in the base, where $\hat{f}$ is a sub-shift of finite type and $\hat{\mu}$ is an ergodic $\hat{f}$-invariant probability measure with local product structure and fully supported. This class of measures includes the equilibrium states of H\"older continuous potentials \cite{Bwn} and fully supported Bernoulli measures when $\hat{f}$ is a Bernoulli shift. 

We remark that $\hat{f}$ needs to be transitive in order to admit a probability measure $\hat{\mu}$ as above. 

\begin{theoA} Let $\hat{A}\in \mathcal{S}_{\alpha}(\Sig,2)$ be a non-uniformly fiber-bunched cocycle which admits a uniform stable holonomy.  If $(\hat{A}_k, H^{s,k})\to (\hat{A},H^s)$ in $\mathcal{H}^s_{\alpha}$, then $\lambda_{+}(\hat{A}_k)\to \lambda_{+}(\hat{A})$. 
\end{theoA}

Let $(\hat{M},\hat{f})$ be a Bernoulli shift. That is, $\hat{M}=\Sig$ and the matrix $Q$ that defines $\Sig$ has all its entries equal to 1. In this case, we consider the measure $\hat{\mu}$ as a fully supported Bernoulli measure.

We say that $\hat{A}\colon \hat{M}\to SL(2,\mathbb{R})$ is a \emph{locally constant cocycle} if it only depends on the zeroth coordinate, that is, $\hat{A}(\hat{x})=A(x_0)$ for some continuous function $A\colon \{1,...,d\}\to SL(2,\mathbb{R})$.

A locally constant cocycle $\hat{A}$ is \emph{irreducible} if there is no proper subspace of $\mathbb{R}^2$ invariant under $A(x_0)$ for every $x_0\in \{1, ..., d\}$.

\begin{theoB} Let $\hat{A}\in \mathcal{S}_{\alpha}(\hat{M},2)$ be a non-uniformly fiber-bunched, locally constant and irreducible cocycle. If $\hat{A}_k\to \hat{A}$ in $\mathcal{S}_{\alpha}(\hat{M},2)$, then $\lambda_{+}(\hat{A}_k)\to \lambda_{+}(\hat{A})$. 
\end{theoB} 

Observe that a locally constant cocycle is $\alpha$-H\"older continuous for every $\alpha>0$. Therefore, Theorem B implies that if $2\lambda_{+}(\hat{A})<\beta \log \sigma$, then $\hat{A}$ is a $C^{\alpha}$-continuity point of the Lyapunov exponents for every $\alpha\geq \beta$. 

In the following we construct an example that verifies the hypotheses of Theorem B. 

\begin{example} Let $\hat{M}=\{0,1\}^{\mathbb{Z}}$, $p=\displaystyle \frac{1}{2}\delta_0 + \frac{1}{2} \delta_1$, $\hat{\mu}=p^{\mathbb{Z}}$ and $\eta> 1$.

Define a locally constant cocycle $\hat{B}\colon \hat{M}\to SL(2,\mathbb{R})$ by the matrices 
$$B_0=\begin{pmatrix}
\eta & 0 \\
0 & \eta^{-1} 
\end{pmatrix}\quad \text{and} \quad B_1=\begin{pmatrix}
0 & 1 \\
-1  & 0 
\end{pmatrix}.$$ This means, $\hat{B}(\hat{x})=B_0$ if $x_0=0$ and $\hat{B}(\hat{x})=B_1$ if $x_0=1$. 

Observe that $\lambda_{+}(\hat{B})=0$. In particular, it is a $C^0$-continuity point for the Lyapunov exponents, see \cite{BV}. 

Consider $$B_{1,n}=\begin{pmatrix}
\cos \theta_n & \sin \theta_n \\
-\sin \theta_n  & \cos \theta_n 
\end{pmatrix},$$ such that $\theta_n\in \mathbb{R}\setminus \mathbb{Q}$ and $\theta_n\to \frac{\pi}{2}$ when $n\to \infty$. 

Let $\hat{B}_n\colon \hat{M}\to SL(2,\mathbb{R})$ be the locally constant cocycle defined by $B_0$ and $B_{1,n}$. Observe that for every $n\in \mathbb{N}$, the cocycle $\hat{B}_n$ is irreducible. In particular, by \cite{F}, $\lambda_{+}(\hat{B}_n)>0$ for every $n\in \mathbb{N}$. 

Since $\hat{B}_n\to \hat{B}$ in the $C^0$ topology, we infer that $\lambda_{+}(\hat{B}_n)\to 0$ as $n\to \infty$. Then, we can choose $n$ big enough such that the cocycle $\hat{B}_n$ is non-uniformly fiber-bunched and therefore satisfies the hypotheses of Theorem B. We conclude that fixed $\alpha>0$ there exists $n$ big enough such that $\hat{B}_n$ is a $C^{\alpha}$-continuity point for the Lyapunov exponents.
\end{example} 

We remark that the value of $\eta$ in Example 1 can be chosen in order to guarantee that the cocycle $\hat{B}_n$ is not fiber-bunched. Therefore, the conclusion of Theorem B applied to Example 1 does not follow from the continuity results of either \cite{BBB} or \cite{BV}. Moreover, Example 1 is a counterexample for a conjecture of Viana about fiber-bunched cocycles, see Conjecture 10.12 in \cite{LLE}. 

\section{Non-uniform invariant holonomies}

In this section, we consider a weaker version of the holonomies introduced in Definition \ref{holonomiafuerte}. In particular, for non-uniformly fiber-bunched cocycles we are able to construct this type of invariant holonomies and to prove that they vary continuously with the cocycle. 

Recall that we consider $\hat{f}$ as a sub-shift of finite type, $\hat{\mu}$ a fully supported ergodic probability measure and $\hat{A}\in \mathcal{S}_{\alpha}(\Sig, 2)$. The hypothesis of $\hat{\mu}$ having local product structure is not used in this section. 

\begin{defn}\label{holonomiadebil} 
A \emph{non-uniform stable holonomy} for $\hat{A}$ is given by a $\hat{\mu}$-full measure set $M^s$ and a collection of linear isomorphisms $H^{s}_{\hat{y},\hat{z}}\colon \mathbb{R}^2\to \mathbb{R}^2$, defined for every $\hat{y},\hat{z}\in W^s_{loc}(\hat{x})$ if $\hat{x}\in M^s$, which satisfy the following properties, 
\vspace{0.2cm}

\begin{enumerate}[label=\emph{(\alph*)}]
\item $H^{s}_{\hat{y},\hat{z}}\circ H^{s}_{\hat{w},\hat{y}}= H^{s}_{\hat{w},\hat{z}}$ and $H^{s}_{\hat{y},\hat{y}}=Id$,
\vspace{0.2cm}

\item $H^{s}_{\hat{f}(\hat{y}),\hat{f}(\hat{z})}=\hat{A}({\hat{z}})\circ H^{s}_{\hat{y},\hat{z}}\circ \hat{A}(\hat{y})^{-1}$, 
\end{enumerate}
\vspace{0.2cm}

\noindent and there exists an increasing sequence $\{\mathcal{D}^s_l\}_{l\in \N}$ of compact subsets such that $\bigcup_l \mathcal{D}^s_l=M^s$ and for every $l\in \mathbb{N}$, $\mathcal{D}^s_l\subset\mathcal{D}^s_{l+1}$, $\hat{\mu}(\Sig\setminus\mathcal{D}^s_l)<\frac{1}{l}$ and, 
\vspace{0.2cm}

\begin{enumerate}[label=\emph{(\alph*)}]
\setcounter{enumi}{2}
\item $(\hat{y}, \hat{z})\mapsto H^{s}_{\hat{y},\hat{z}}$ is continuous in $\mathcal{D}_l^s$ for $(\hat{y},\hat{z})\in\Omega^s$.
\end{enumerate}
\vspace{0.2cm}

A \emph{non-uniform unstable holonomy} for $\hat{A}$ is defined analogously for points in the same local unstable set. We use the expression \emph{non-uniform invariant holonomies} to refer to both non-uniform stable and non-uniform unstable holonomies.
\end{defn}

More precisely, item (c) states that for every $l\in \mathbb{N}$ and every $\epsilon>0$ there exists $\delta_l>0$ such that $$\|H^s_{\hat{y}_1,\hat{z}_1}-H^s_{\hat{y}_2,\hat{z}_2}\|< \epsilon,$$ if there exist $\hat{x}_1,\hat{x}_2\in \mathcal{D}^s_l$ such that $\hat{y}_1,\hat{z}_1\in W^s_{loc}(\hat{x}_1)$, $\hat{y}_2,\hat{z}_2\in W^s_{loc}(\hat{x}_2)$ and $$d(\hat{y}_1,\hat{z}_1)<\delta_l\quad \text{and} \quad d(\hat{y}_2,\hat{z}_2)<\delta_l.$$ Observe that the continuity of $H^s$ depends on the set $\mathcal{D}^s_l$, therefore the constant $\delta_l>0$ can be decreasing when $l$ increases. This is one of the main difficulties when working with Definition \ref{holonomiadebil}. 

In particular, item (c) implies that for every $\xi\in \mathbb{R}^2$, $$(\hat{y}, \hat{z}, \xi)\mapsto H^{s}_{\hat{y},\hat{z}}(\xi) \text{ is measurable in } M^s \text{ for } (\hat{y},\hat{z})\in\Omega^s.$$ 

We call invariant holonomies both the uniform invariant holonomies as in Definition \ref{holonomiafuerte} and the non-uniform invariant holonomies as above. 

From now, in order to simplify the notation we restrict to the case $\alpha=1$, that is, $\hat{A}\in \mathcal{S}_1(\Sig, 2)$ and therefore, $\hat{A}$ is a non-uniformly fiber-bunched cocycle if $2\lambda_{+}(\hat{A})<\log \sigma.$ We denote $\mathcal{H}^s_{1}=\mathcal{H}^s$. The general case for $0<\alpha<1$ is analogous. 

We use the next definition to prove the existence of non-uniform invariant holono\-mies. This is a consequence of the results in \cite{V}. 

\begin{defn}\label{dom} Given $N\in \mathbb{N}$ and $\theta>0$, define $\mathcal{D}^s_{\hat{A}}(N,\theta)$ as the set of points $\hat{x}$ satisfying $$\prod_{j=0}^{k-1}\|\hat{A}^{N}(\hat{f}^{jN}(\hat{x}))\|\|\hat{A}^{N}(\hat{f}^{jN}(\hat{x}))^{-1}\|\leq e^{kN\theta}\text{ for all }k\geq1.$$ 
Analogously, we define $\mathcal{D}^u_{\hat{A}}(N,\theta)$ when the inequality is satisfied by $(\hat{f}^{-1}, \hat{A}^{-1})$ instead.
\end{defn}

It is possible to construct the linear isomorphisms $H^s_{\hat{y},\hat{z}}$ as in Definition \ref{holonomiadebil} for the elements in $\mathcal{D}^s_{\hat{A}}(N,\theta)$.

Recall that $\rho$ is a fixed constant associated to the distance $d_{\rho}$. Moreover, $\rho$ is the contraction rate of $\hat{f}$ and $\sigma=1/\rho$ is the expansion rate of $\hat{f}$. 

\begin{prop}[Proposition 2.5, \cite{V}]\label{exist} Let $\hat{A}\in \mathcal{S}_1(\Sig,2)$. Given $N\in \mathbb{N}$ and $ \theta>0$ with $\theta<\log \sigma$, there exists $C=C(N,\theta)>0$ such that for every $\hat{x}\in\mathcal{D}^s_{\hat{A}}(N,\theta)$ and $\hat{y},\hat{z}\in W^s_{loc}(\hat{x})$, $$H^s_{\hat{y},\hat{z}}=\lim _{n\to +\infty} \hat{A}^{n}(\hat{z})^{-1}\hat{A}^{n}(\hat{y})$$ exists and satisfies $\|H^s_{\hat{y},\hat{z}}-Id\|\leq C d(\hat{y},\hat{z})$.\end{prop}
\textit{Proof.}
Define $H^n_{\hat{y},\hat{z}}=\hat{A}^{n}(\hat{z})^{-1}\hat{A}^{n}(\hat{y})$. In order to prove that the limit exists, it is enough to demonstrate that $\{H^n_{\hat{y},\hat{z}}\}$ is a Cauchy sequence. 

By Lemma 2.6 of \cite{V} we deduce that there exist $C_0=C_0(\hat{A},N)>0$ such that $$\|\hat{A}^n(\hat{y})\|\|\hat{A}^n(\hat{z})^{-1}\|\leq C_0e^{n\theta},$$ for every $\hat{y},\hat{z}\in W^s_{loc}(\hat{x})$, $\hat{x}\in \mathcal{D}^s_{\hat{A}}(N,\theta)$ and $n\geq 0$. This is a consequence of Definition \ref{dom}. 

Therefore, 
\begin{equation}\label{cauchy} 
\begin{aligned}
\| H^{n+1}_{\hat{y},\hat{z}}- H^{n}_{\hat{y},\hat{z}}\|&=\|\hat{A}^{n+1}(\hat{z})^{-1}\circ \hat{A}^{n+1}(\hat{y})-\hat{A}^{n}(\hat{z})^{-1}\circ \hat{A}^{n}(\hat{y})\|\\
&\leq \|\hat{A}^{n}(\hat{z})^{-1}\|\|\hat{A}(\hat{f}^n(\hat{z}))^{-1}\circ \hat{A}(\hat{f}^n(\hat{y}))-Id\|\|\hat{A}^{n}(\hat{y})\|\\
&\leq C_1 e^{n\theta}\text{dist}(\hat{f}^n(\hat{y}),\hat{f}^n(\hat{z}))\\
&\leq C_1 e^{n(\theta - \log\sigma)} \text{dist}(\hat{y},\hat{z}),
\end{aligned}
\end{equation}
here $C_1$ depends on the Lipschitz constant of $\hat{A}$. Since $\theta< \log \sigma$, we conclude that $\{H^s_{\hat{y},\hat{z}}\}$ is a Cauchy sequence. 

Observe that Equation (\ref{cauchy}) implies the following inequality, 

\begin{equation}\label{Hn}
\begin{aligned}
\| H^{n}_{\hat{y},\hat{z}}- H^{s}_{\hat{y},\hat{z}}\|&\leq\sum_{j=n}^\infty\|\hat{A}^{j+1}(\hat{z})^{-1}\circ \hat{A}^{j+1}(\hat{y})-\hat{A}^{j}(\hat{z})^{-1}\circ \hat{A}^{j}(y)\|\\
&\leq C_1 d(\hat{y},\hat{z})\sum_{j=n}^\infty e^{j(\theta-\log\sigma)}\leq Cd(\hat{y},\hat{z})e^{n(\theta-\log \sigma)}.
\end{aligned}
\end{equation}
In particular, $\|H^s_{\hat{y},\hat{z}}-Id\|\leq C d(\hat{y},\hat{z})$. $\square$ 

The next proposition gives sufficient conditions to guarantee the existence of non-uniform invariant holonomies. 

\begin{prop}\label{existdebil}
If $\hat{A}\in \mathcal{S}_1(\Sig,2)$ is a non-uniformly fiber-bunched cocycle, then $\hat{A}$ admits non-uniform invariant holonomies. 
\end{prop}
\textit{Proof.}
Corollary 2.4 in \cite{V} states that if $\theta$ verifies $$2\lambda_{+}(\hat{A})<\theta<\log\sigma, \text{ then } \hat{\mu}\left(\bigcup_{N=1}^{\infty} \mathcal{D}^s_{\hat{A}}(N, \theta)\right)=1.$$

Moreover, the subsets $\mathcal{D}^s_{\hat{A}}(N,\theta)$ satisfy:
\begin{enumerate}[label=(\alph*)]
\item $\mathcal{D}^s_{\hat{A}}(N, \theta)$ is closed, then compact. 
\vspace{0.2cm}

\item $\mathcal{D}^s_{\hat{A}}(N, \theta)\subset \mathcal{D}^s_{\hat{A}}(lN, \theta)$ for each $l \geq 1$. 
\end{enumerate}

Therefore, we can define a sequence of compact subsets $\{\mathcal{D}^s_{\hat{A},l}\}$ such that $\mathcal{D}^s_{\hat{A},l}\subset \mathcal{D}^s_{\hat{A},l+1}$ and $\hat{\mu}(\mathcal{D}^s_{\hat{A},l})\to 1$ when $l\to \infty$. In order to verify this, we observe that for each $l\in \mathbb{N}$ there exists $k_l$ such that $$\hat{\mu}\left(\bigcup_{N=1}^{k_l} \mathcal{D}^s_{\hat{A}}(N,\theta)\right) > 1 - \frac{1}{l}.$$ Take $N_l=k_l! N_{l-1}$, then consider $\mathcal{D}^s_{\hat{A},l}=\mathcal{D}^s_{\hat{A}}(N_l, \theta)$. 

Let $M^s=\bigcup_l \mathcal{D}^s_{\hat{A},l}$, then by Proposition \ref{exist}, we conclude that the sets $\mathcal{D}^s_{\hat{A},l}$ verify all properties in Definition \ref{holonomiadebil}. Observe that the continuity required in item (c) follows by Equation (\ref{Hn}). 

If we apply the same argument to $(\hat{f}^{-1},\hat{A}^{-1})$, we deduce that there exist subsets $\mathcal{D}^u_{\hat{A},l}$ that allow us to conclude that there also exists a non-uniform unstable holonomy for $\hat{A}$. $\square$ 

We remark that the non-uniform invariant holonomies constructed in Proposition \ref{existdebil} satisfy stronger properties that the ones in Definition \ref{holonomiadebil}. By Proposition \ref{exist}, we infer that \emph{there exists $C_l>0$ such that $\|H^s_{\hat{y},\hat{z}}-Id\|\leq C_l d(\hat{y},\hat{z})$ for every $\hat{y}, \hat{z}\in W^s_{loc}(\hat{x})$ if $\hat{x}\in \mathcal{D}^s_{\hat{A},l}$}. 

\begin{defn}\label{block}
We call \emph{stable holonomy blocks for} $\hat{A}$ to the increasing sequence of compact sets $\mathcal{D}^s_{\hat{A},l}=\mathcal{D}^s_{\hat{A}}(N_l,\theta)$ as in the proof of Proposition \ref{existdebil}. Analogously, we refer to $\{\mathcal{D}^u_{\hat{A},l}\}$ as \emph{unstable holonomy blocks for} $\hat{A}$. 
\end{defn}

Suppose $\hat{A}$ satisfies the hypotheses of Proposition \ref{existdebil} and $\{\hat{A}_k\}_{k\in\N}$ is a sequence of cocycles such that $\hat{A}_k\to \hat{A}$ in the Lipschitz topology. Since, $\lambda_+$ is upper semi-continuous, then \begin{equation}\label{Ak}\limsup_{k\to\infty}\; 2\lambda_+(\hat{A}_k)\leq 2\lambda_+(\hat{A})< \log\sigma.
\end{equation}  Therefore, we also are able to prove the existence of non-uniform invariant holonomies for every $\hat{A}_k$.

In the following proposition we show that the non-uniform invariant holo\-no\-mies, given by Proposition \ref{existdebil}, are continuous as a function of $\hat{A}$. The precise statement and proof are given for the non-uniform stable holonomy, but they are analogous for the non-uniform unstable one. 

\begin{prop} \label{holonomia}
If $\hat{A}\in \mathcal{S}_1(\Sig,2)$ is a non-uniformly fiber-bunched cocycle, $\hat{A}_k\to \hat{A}$ in the Lipschitz topology and $\hat{x}\in\mathcal{D}^s_{\hat{A},l}$, then there exists $k_l\in \mathbb{N}$ such that if $k\geq k_l$, $H^{s,k}_{\hat{y},\hat{z}}$ exists for all $\hat{y},\hat{z}\in W^{s}_{loc}(\hat{x})$ and satisfies  $H^{s,k}_{\hat{y},\hat{z}}\to H^{s}_{\hat{y},\hat{z}}$. Furthermore, the convergence is uniform in $\mathcal{D}^s_{\hat{A},l}$.
\end{prop}
\textit{Proof.}
From Definition \ref{dom}, fixed $l\in \mathbb{N}$, we know that there exists an open neighborhood $\mathcal{U}^l$ of $\hat{A}$ in the Lipschitz topology and a constant $\hat{\theta}>\theta$ such that if $\hat{x}\in \mathcal{D}^s_{\hat{A},l}=\mathcal{D}^s_{\hat{A}}(N_l, \theta)$, then $\hat{x}\in \mathcal{D}^s_{\hat{B}}(N_l, \hat{\theta})$ for all $\hat{B}\in \mathcal{U}^l$. We assume that $\theta$ in Definition \ref{block} was taken to be uniform in $\mathcal{U}^l$, that is $\theta=\hat{\theta}$. 

For any sequence $\{\hat{A}_k\}_{k\in\N}$ converging to $\hat{A}$ in the Lipschitz topology, there exists $k_l$ such that $\hat{A}_k\in\mathcal{U}^l$ if $k\geq k_l$. Therefore, by Proposition \ref{exist} the holonomies $H^{s,k}_{\hat{y},\hat{z}}$ are defined for every $\hat{A}_k$ with $k\geq k_l$ and for any $\hat{x}\in \mathcal{D}^s_{\hat{A},l}$.

Let $n\in \N$ and $k\geq k_l$. Define $$H^n_{\hat{y},\hat{z}}=\hat{A}^{n}(\hat{z})^{-1}\circ \hat{A}^{n}(\hat{y})\quad \text{and}\quad H^{n,k}_{\hat{y},\hat{z}}=\hat{A}_k^{n}(\hat{z})^{-1}\circ \hat{A}_k^{n}(\hat{y}).$$ 

By Equation (\ref{Hn}), we know that there exists $C_0>0$ and $C_k>0$ such that, $$\|H^{s}_{\hat{y},\hat{z}}-H^{n}_{\hat{y},\hat{z}}\|\leq C_0 d(\hat{y},\hat{z})e^{n(\theta-\log\sigma)} \quad \text{ and } \quad \|H^{s,k}_{\hat{y},\hat{z}}-H^{n,k}_{\hat{y},\hat{z}}\|\leq C_k d(\hat{y},\hat{z})e^{n(\theta-\log\sigma)}.$$

In the equation above the constants $C_0$ and $C_k$ depend on $l$ and the Lipschitz constant of $\hat{A}$ and $\hat{A}_k$, respectively. Therefore, they can be chosen to be uniform in $\mathcal{U}^l$. 

Summarizing, there exists $C> 0$ such that, 
\begin{equation*}
\begin{aligned}
\|H^{s,k}_{\hat{y},\hat{z}}-H^s_{\hat{y},\hat{z}}\|&\leq \|H^{s,k}_{\hat{y},\hat{z}}-H^{n,k}_{\hat{y},\hat{z}}\|+\|H^{n,k}_{\hat{y},\hat{z}}-H^{n}_{\hat{y},\hat{z}}\|+\|H^{n}_{\hat{y},\hat{z}}-H^s_{\hat{y},\hat{z}}\|\\
&\leq 2Cd(\hat{y},\hat{z})e^{n(\theta-\log\sigma)}+\|H^{n,k}_{\hat{y},\hat{z}}-H^{n}_{\hat{y},\hat{z}}\|.
\end{aligned}
\end{equation*}
Since $H^{n}_{\hat{y},\hat{z}}$ varies continuously with the cocycle, the proposition follows. $\square$

\section{Invariance Principle}

One of the main tools in the proof of our results is the Invariance Principle, which was first developed by Furstenberg \cite{F} and Ledrappier \cite{L} for random matrices and was extended by Bonatti, G\'omez-Mont, Viana \cite{BGV} to linear cocycles over hyperbolic systems and by Avila, Viana \cite{AV} and Avila, Santamaria, Viana \cite{ASV} to general (diffeomorphisms) cocycles. In \cite{AV} the base dynamics is still assumed to be hyperbolic, whereas in \cite{ASV}, it is taken to be partially hyperbolic and volume-preserving.

In the following, we state a version of the Invariance Principle for non-uniform invariant holonomies. This context has been considered before in \cite{V} for a single stable holonomy block $\mathcal{D}^s_{\hat{A},l}$ as in Definition \ref{block}. 

Given $\hat{A}\in \mathcal{S}_1(\Sig,2)$ define the projectivization of the skew-product $F_{\hat{A}}$, $$\mathbb{P}(F_{\hat{A}})\colon \Sig\times \mathbb{P}^1\to \Sig\times \mathbb{P}^1\; \text{ by }\;\mathbb{P}(F_{\hat{A}})(\hat{x},[v])=(\hat{f}(\hat{x}),[\hat{A}(\hat{x})v]).$$ In order to simplify the notation, we denote $\mathbb{P}(F_{\hat{A}})=\mathbb{P}(\hat{A})$ and $[v]=\mathbb{P}(v)$. 

Let $\pi\colon \Sig\times \mathbb{P}^1\to \Sig$ be the canonical projection to the first coordinate. We study the $\mathbb{P}(\hat{A})$-invariant probability measures $\hat{m}$ such that $\pi_*\hat{m}=\hat{\mu}$. Since $\mathbb{P}(\hat{A})$ is a continuous map defined on a compact space, these measures always exist.

In this context, by Rokhlin \cite{R}, there exists a disintegration of $\hat{m}$ into conditional probabilities $\{\hat{m}_{\hat{x}}\}_{ \hat{x}\in \Sig}$ along the fibers which is essentially unique, that is, a measurable family of probability measures such that $\hat{m}_{\hat{x}}(\{\hat{x}\}\times\mathbb{P}^1)=1$ for $\hat{\mu}$-almost every $\hat{x}\in\Sig$ and $$\hat{m}(E)=\int \hat{m}_{\hat{x}}(E\cap \left(\{\hat{x}\}\times\mathbb{P}^1\right))d\hat{\mu}$$ for every measurable set $E\subset \Sig\times\mathbb{P}^1$.

If $H^{s}$ and $H^{u}$ are invariant holonomies for $\hat{A}$, we can define invariant holonomies for $\mathbb{P}(\hat{A})$ as $h^{s}_{\hat{x},\hat{y}}=\mathbb{P}(H^{s}_{\hat{x},\hat{y}})$ and $h^{u}_{\hat{x},\hat{y}}=\mathbb{P}(H^{u}_{\hat{x},\hat{y}})$. The following definition gives a relation between these holonomies and the disintegration of $\hat{m}$.

\begin{defn}\label{sugeral} Let $h^u$ be a unstable holonomy for $\mathbb{P}(\hat{A})$ and $\hat{m}$ be a $\mathbb{P}(\hat{A})$-invariant probability measure projecting to $\hat{\mu}$. 

We say that $\hat{m}$ is a \emph{$u$-state} if there exist a disintegration $\{\hat{m}_{\hat{x}}\}_{ \hat{x}\in \Sig}$ and a $\hat{\mu}$-full measure set $M^u$ such that $(h^{u}_{\hat{x},\hat{y}})_{*}\hat{m}_{\hat{x}}=\hat{m}_{\hat{y}}$ for every $\hat{x},\hat{y}\in M^u$ in the same local unstable set. 

The definition of \emph{$s$-state} is stated analogously. If a measure is simultaneously a $u$-state and an $s$-state, we call it \emph{$su$-state}.
\end{defn}

Observe that the definition above depends only of a full measure set of $\Sig$, then the same definition can be used for uniform invariant holonomies as in Definition \ref{holonomiafuerte} or non-uniform invariant holonomies as in Definition \ref{holonomiadebil}. 

Let $\hat{A}\in \mathcal{S}_1(\Sig,2)$. If $\lambda_{+}(\hat{A})=0$, then $\hat{A}$ is a non-uniformly fiber-bunched cocycle and therefore by Proposition \ref{existdebil} it admits non-uniform invariant holonomies $H^s$ and $H^u$. Moreover, since $\lambda_{+}(\hat{A})=0$, a stronger version of Proposition \ref{exist} can be deduced: 

\emph{For every $l\in \mathbb{N}$ there exists $C_l>0$ such that for every $\hat{x}\in\mathcal{D}^s_{\hat{A},l}$, $\hat{y},\hat{z}\in W^s_{loc}(\hat{x})$ and $j\geq 0$, $\|H^s_{\hat{f}^j(\hat{y}),\hat{f}^j(\hat{z})}-Id\|\leq C_l d(\hat{y},\hat{z})$.} 

We refer the reader to Corollary 2.8 of \cite{V} for a proof of the property mentioned above. We remark that the existence of a uniform upper bound for all the iterates is a necessary condition in order to be able to prove the next proposition. 

Recall that $\mathcal{D}^s_{\hat{A},l}$ has been described in Definition \ref{block}. 
 
\begin{prop}[Proposition 3.1, \cite{V}]\label{invarianzav} Let $\hat{A}\in \mathcal{S}_1(\Sig,2)$ be such that $\lambda_+(\hat{A})=\lambda_{-}(\hat{A})=0$ and let $\hat{m}$ be any $\mathbb{P}(\hat{A})$-invariant probability measure that projects to $\hat{\mu}$. Then, there exists a full $\hat{\mu}$-measure subset $E^s_l$ of $\mathcal{D}^s_{\hat{A},l}\cap[0;i]$, for every $i \in\{1, ..., d\}$, such that the disintegration $\{\hat{m}_{\hat{z}}\}$ of $\hat{m}$ satisfies $$\hat{m}_{\hat{z}_2}=\left(h^s_{\hat{z}_1,\hat{z}_2}\right)_*\hat{m}_{\hat{z}_1}$$ for every $\hat{z}_1,\hat{z}_2 \in E^s_l$ in the same local stable set. 
\end{prop}

Replacing $(\hat{f},\hat{A})$ by $(\hat{f}^{-1},\hat{A}^{-1})$ we infer the existence of a disintegration which is invariant under the non-uniform unstable holonomy over a full $\hat{\mu}$-measure subset $E^u_l$ in $\mathcal{D}^u_{\hat{A},l}\cap[0;i]$.

\begin{theo}\label{invarianza} If $\hat{A}\in \mathcal{S}_1(\Sig,2)$ and $\lambda_+(\hat{A})=\lambda_{-}(\hat{A})=0$, then every $\mathbb{P}(\hat{A})$-invariant probability measure $\hat{m}$ projecting to $\hat{\mu}$ is an $su$-state. \end{theo}
\textit{Proof.}
By Proposition \ref{invarianzav}, we deduce that for each $i\in \{1,..,d\}$  and $l\in \mathbb{N}$ there is a disintegration $\{\hat{m}_{\hat{z}}\}$ and a full measure subset $E^s_l$ of $\mathcal{D}^s_{\hat{A},l}\cap[0;i]$, that satisfy $$\hat{m}_{\hat{z}_2}=\left(h^s_{\hat{z}_1,\hat{z}_2}\right)_*\hat{m}_{\hat{z}_1}$$ for every $\hat{z}_1,\hat{z}_2 \in E^s_l$ in the same local stable set. 

We want to find a disintegration $\{\hat{m}_{\hat{z}}\}$ invariant by the non-uniform stable holonomy in a full measure set of $\Sig$. 

Initially, we fix $[0;i]$ and denote $\{\hat{m}^l_{\hat{z}}\}$ the disintegration of $\hat{m}$ associated to $\mathcal{D}^s_{\hat{A},l}\cap[0;i]$. By the essential uniqueness of disintegrations, we conclude that $\hat{m}^l_{\hat{z}}=\hat{m}^{l+1}_{\hat{z}}$ for $\hat{\mu}$-almost every $\hat{z}\in\mathcal{D}^s_{\hat{A},l}\cap [0;i]\subset \mathcal{D}^s_{\hat{A},l+1}\cap [0;i]$. Then, for every $l\in \mathbb{N}$, we denote by $E^s_l$ the full measure subset of $\mathcal{D}^s_{\hat{A},l}\cap[0;i]$, satisfying that $E^s_l\subset E^s_{l+1}$ and $\hat{m}^l_{\hat{z}}=\hat{m}^{l+1}_{\hat{z}}$ for $\hat{z}\in E^s_l$. Next, we define a disintegration of $\hat{m}$ on each $[0;i],$ $$\hat{m}_{\hat{z}}=\left\lbrace
           \begin{matrix}
            \hat{m}^1_{\hat{z}} & \hat{z}\in E^s_1\subset\mathcal{D}^s_{\hat{A},1},\\
            \hat{m}^{l+1}_{\hat{z}} & \hat{z}\in E^s_{l+1}\setminus E^s_l.
           \end{matrix}
         \right.$$
         
Let $E^s=\bigcup E^s_l$. In order to prove the invariance of the disintegration by the non-uniform stable holonomy in $E^s$, it is enough to choose $\hat{z}_1,\hat{z}_2\in E^s$ in the same stable leaf and verify $\hat{m}_{\hat{z}_2}=\left(h^s_{\hat{z}_1,\hat{z}_2}\right)_*\hat{m}_{\hat{z}_1}$. As $\{E^s_l\}$ is an increasing sequence of sets, there exists $l$ such that $\hat{z}_1,\hat{z}_2\in E^s_l$, and then, the definition of the disintegration implies $$\hat{m}_{\hat{z}_2}=\hat{m}_{\hat{z}_2}^l=\left(h^s_{\hat{z}_1,\hat{z}_2}\right)_*\hat{m}^l_{\hat{z}_1}=\left(h^s_{\hat{z}_1,\hat{z}_2}\right)_*\hat{m}_{\hat{z}_1}.$$

Proceeding this way in every $[0;i]$, we infer the existence of a disintegration in $\hat{\mu}$-almost every point in $\Sig$. We conclude that $\hat{m}$ is an $s$-state. Applying the same argument to $(\hat{f}^{-1}, \hat{A}^{-1})$ we deduce that $\hat{m}$ is also a $u$-state. $\square$ 

Theorem \ref{invarianza} is also true if the cocycle $\hat{A}$ admits one uniform invariant holonomy and one non-uniform invariant holonomy. In this case, the proof follows from the argument above and Proposition 1.16 of \cite{BGV}. 

The following proposition is essential to prove Theorem A and the hypothesis of $\hat{\mu}$ having local product structure is crucial here. Both the uniform stable holonomy and the local product structure of $\hat{\mu}$ allow us to use the non-uniform unstable holonomy to transport the disintegration of $\hat{m}$ from a local unstable set to every point of the cylinder in a continuous way. The result is an extension of Theorem 6 in \cite{BGV}, following the ideas of Proposition 4.8 of \cite{AV}.

\begin{prop}\label{yanise}
Let $\hat{A}\in \mathcal{S}_1(\Sig,2)$ be a non-uniformly fiber-bunched cocycle which admits a uniform stable holonomy. If $\hat{m}$ is an $su$-state, then there exists a continuous disintegration of $\hat{m}$, $\{\hat{m}_{\hat{z}}\}$, which is invariant by $\mathbb{P}(\hat{A})$ and both holonomies. 
\end{prop}
\textit{Proof.}
We start by considering the non-uniform unstable holonomy given by Proposition \ref{existdebil}. By definition of $su$-state, there exist two disintegrations $\{\hat{m}_{\hat{x}}^1\}_{\hat{x}\in\Sig}$ and $\{\hat{m}_{\hat{x}}^2\}_{\hat{x}\in\Sig}$ of $\hat{m}$, and a $\hat{\mu}$-full measure subset $\hat{U}_{i}$ of $[0;i]$ for every $i\in\{1,...,d\}$ such that
\vspace{0.2cm}

\begin{enumerate}[label=(\roman*)]
\item $(h^s_{\hat{x},\hat{y}})_*\hat{m}_{\hat{x}}^1=\hat{m}_{\hat{y}}^1$ for each $\hat{y}\in W^s_{loc}(\hat{x})$ with $\hat{x}\in\hat{U}_{i}$ ($s$-state);
\vspace{0.2cm}

\item $(h^u_{\hat{x},\hat{y}})_*\hat{m}_{\hat{x}}^2=\hat{m}_{\hat{y}}^2$ for each $\hat{y}\in W^u_{loc}(\hat{x})$ with $\hat{x}\in\hat{U}_{i}$ ($u$-state);
\vspace{0.2cm}

\item $\hat{m}_{\hat{x}}^1=\hat{m}_{\hat{x}}^2$ for each $\hat{x}\in\hat{U}_{i}$ (essential uniqueness of disintegration).
\end{enumerate}
\vspace{0.2cm}

We consider $l$ large enough such that the unstable holonomy block $\mathcal{D}^u_{A,l}$ of Definition \ref{block} satisfies $\mathcal{D}^u_{\hat{A},l}\cap\hat{U}_i\neq \emptyset$ for each $i\in \{1,...,d\}$ and fix $\hat{x}\in \mathcal{D}^u_{\hat{A},l}\cap\hat{U}_{i}$ such that $\mu^u_{\hat{x}}(W^u_{loc}(\hat{x})\setminus\hat{U}_{i})=0$. Here $\mu^u_{\hat{x}}$ is the element of the disintegration of $\hat{\mu}$ relative to the unstable sets of $\hat{f}$. See the definitions before Lemma \ref{lpe}. 

Define $\hat{m}_{\hat{x}}=\hat{m}_{\hat{x}}^1$ and
\vspace{0.2cm}

\begin{enumerate}[label=(\alph*)]
\item $\hat{m}_{\hat{y}}=(h^u_{\hat{x},\hat{y}})_*\hat{m}_{\hat{x}}=(h^u_{\hat{x},\hat{y}})_*\hat{m}_{\hat{x}}^1$ for each $\hat{y}\in W^u_{loc}(\hat{x})\cap [0;i]$;
\vspace{0.2cm}

\item $\hat{m}_{\hat{z}}=(h^s_{\hat{y},\hat{z}})_*\hat{m}_{\hat{y}}$ for each $\hat{z}\in W^s_{loc}(\hat{y})\cap [0;i]$ with $\hat{y}\in W^u_{loc}(\hat{x})\cap [0;i]$.
\end{enumerate}
\vspace{0.2cm}

By (i)-(iii), we conclude that $\hat{m}_{\hat{y}}=\hat{m}_{\hat{y}}^1=\hat{m}_{\hat{y}}^2$ for every $\hat{y}\in W^u_{loc}({\hat{x}})\cap \hat{U}_{i}$ and  $\hat{m}_{\hat{z}}=\hat{m}_{\hat{z}}^2$ for every $\hat{z}\in W^s_{loc}(\hat{y})\cap\hat{U}_{i}$ with $\hat{y}\in  W^u_{loc}({\hat{x}})\cap \hat{U}_{i}$. By the choice of $\hat{x}$ and the fact that $\hat{\mu}$ has local product structure, the later corresponds to a  full measure subset of points $\hat{z}\in[0;i]$. In particular, $\{\hat{m}_{\hat{z}}\}_{\hat{z}\in[0;i]}$ is a disintegration for $\hat{m}$.

The continuity of $\hat{m}_{\hat{z}}$ is a consequence of the fact that $\hat{z}\mapsto h^{us}_{\hat{x},\hat{z}}$ is a continuous map, where $$h^{us}_{{\hat{x}},{\hat{z}}}=h^{s}_{{\hat{y}},{\hat{z}}}\circ h^{u}_{\hat{x},{\hat{y}}},$$ with $\hat{x},\hat{y},\hat{z}$ as in (a)-(b). This is true because $\hat{y}\mapsto h^u_{\hat{x},\hat{y}}$ is continuous in $W^u_{loc}(\hat{x})$ due to $\hat{x}\in\mathcal{D}^u_{\hat{A},l}$, $\hat{z}\mapsto W^s_{loc}(\hat{z})\cap W^u_{loc}(\hat{x})$ is continuous in $[0;i]$ and $(\hat{z},\hat{y})\mapsto h^s_{\hat{y},\hat{z}}$ is continuous in $\Omega^s$.

By the definition of the disintegration it is clear that it is invariant by both holo\-no\-mies. $\square$

\section{Limit of su-states}

In this section we prove that being an $su$-state is a closed property in the set of non-uniformly fiber-bunched cocycles with the Lipschitz topology. This is used in the next section to give a characterization of discontinuity points of the Lyapunov exponents.  

This property of $su$-states has already been proved in several contexts, for example, see Proposition 5.17 of \cite{LLE} for a proof for locally constant cocycles and Lemma 4.3 of \cite{BBB} and Corollary 2.3 of \cite{TY} for linear cocycles over hyperbolic maps. For linear cocycles over partially hyperbolic diffeomorphisms, it has been stated in Corollary 5.3 of \cite{AV} and a detailed proof can be found in Appendix A of \cite{M}. In these results the holonomies that are considered are uniform invariant holonomies. 

Lemma 4.3 of \cite{BBB} is more general that the next proposition. They allow the measure in the base to variate by taking a sequence $\hat{\mu}_k$. Since we do not have a complete control of the non-uniform invariant holonomies, we are not able to adapt their argument to our context. We use the proof in \cite{M} to conclude a result for non-uniformly fiber-bunched cocycles.
 
We prove that being and $s$-state is a closed property. For $u$-states, it is sufficient to consider $(\hat{f}^{-1}, \hat{A}^{-1}).$

\begin{prop}\label{utosdebil} Let $\hat{A}\in \mathcal{S}_1(\Sig,2)$ be a non-uniformly fiber-bunched cocycle. Suppose that $\hat{A}_k$ converges in the Lipschitz topology to $\hat{A}$ and let $\hat{m}^k$ be an $s$-state for $\mathbb{P}(\hat{A}_k)$ projecting to $\hat{\mu}$. If $\hat{m}^k$ converges to $\hat{m}$ in the weak-$^{*}$ topology, then $\hat{m}$ is an $s$-state for $\mathbb{P}(\hat{A})$.	\end{prop}

\textit{Proof.} By the semi-continuity of $\lambda_+$, if $\hat{A}$ is a non-uniformly fiber-bunched cocycle, then every $\hat{A}_k$ has a non-uniform stable holonomy defined in the stable holonomy blocks $\mathcal{D}^s_{\hat{A},l}$ of Definition \ref{block}. See Proposition \ref{holonomia}.

We begin by fixing $l_0\in \mathbb{N}$ large enough such that $\mathcal{D}^s_{\hat{A},l_0}\cap[0;i]\neq\emptyset$ for each $i\in \{1,...,d\}$. By the construction of the sets $\mathcal{D}^s_{\hat{A},l}$, the same property is true for every $l\geq l_0$. 

We fix $\hat{x}_i\in[0;i]$ for each $i\in \{1,...,d\}$. For every $l\geq l_0$, define $$\mathcal{D}_{i,l}=\{\hat{z}\in W^u_{loc}(\hat{x}_i):  W^s_{loc}(\hat{z})\cap\mathcal{D}^s_{\hat{A},l}\neq \emptyset\}.$$ As a result, $\mathcal{D}^s_{\hat{A},l}\cap [0;i]\subset \bigcup \{W^s_{loc}(\hat{z}) : \hat{z}\in\mathcal{D}_{i,l}\}$. 

Let $\mathcal{B}^l\subset \mathcal{M}$ be the sub $\sigma$-algebra generated by $$\{W^s_{loc}(\hat{z}) : \hat{z}\in \mathcal{D}_{i,l} \text{ and } \ i\in \{1,...,d\}\}.$$ Therefore, the elements in $\mathcal{B}^l$ are the measurable sets $E$ such that for each $\hat{z}$ and $j$, either $E$ contains $W^s_{loc}(\hat{z})$ or it is disjoint from it. 

We fix $l\geq l_0$ and for each $k\in \mathbb{N}$ define the map $$\overline{h}^k_{\hat{y}}=\left\lbrace
           \begin{matrix}
h^{s,k}_{\hat{y},\hat{z}} &\text{ if } \hat{y}\in W^s_{loc}(\hat{z}) \text{ with } \hat{z}\in\mathcal{D}_{i,l}\\
            Id & \text{otherwise}.
           \end{matrix}
         \right.$$
Here $h^{s,k}$ is the projectivization of $H^{s,k}$ and hence it is an stable holonomy for $\mathbb{P}(F_{\hat{A}_k})$. By the definition of the sets $\mathcal{D}_{i,l}$ and Proposition \ref{existdebil}, we know that $\overline{h}^k$ is well defined.

Analogously, we define $\overline{h}$. By Proposition \ref{holonomia}, we infer that $\overline{h}^k_{\hat{y}}$ converges to $\overline{h}_{\hat{y}}$ uniformly on $\Sig$.

Consider $\overline{m}^k_{\hat{y}}=(\overline{h}^k_{\hat{y}})_*\hat{m}^k_{\hat{y}}$ and $\overline{m}_{\hat{y}}=(\overline{h}_{\hat{y}})_*\hat{m}_{\hat{y}}$ for every $\hat{y}\in \Sig$ and, $$\overline{m}^k=\int \overline{m}^k_{\hat{y}}\; d\hat{\mu}\quad \text{ and } \quad \overline{m}=\int \overline{m}_{\hat{y}}\; d\hat{\mu}.$$

We want to prove that $\overline{m}^k\to\overline{m}$ in the weak-$^*$ topology. Then, we have to show that $\int\varphi\, d\overline{m}^k\to\int\varphi\, d\overline{m}$ for every continuous and bounded function $\varphi\colon\Sig\times\mathbb{P}^1\to\R$.

By the definition of $\overline{m}^k$, we deduce that $$\int \varphi\, d\overline{m}^k=\int \int \varphi(\hat{y},\overline{h}^{k}_{\hat{y}}(v))\, d\hat{m}^k_{\hat{y}}(v)d\hat{\mu}.$$ Denote $\overline{\varphi}_k(\hat{y},v)=\varphi(\hat{y},\overline{h}^{k}_{\hat{y}}(v))$ then,
\begin{equation}\label{22}
\begin{aligned}
\left|\int \varphi\, d\overline{m}^k-\int\varphi\, d\overline{m}\right|&=\left|\int \overline{\varphi}_k\, d\hat{m}^k-\int\overline{\varphi}\, d\hat{m}\right|\\
&\leq\int\left| \overline{\varphi}_k-\overline{\varphi}\right|\, d\hat{m}^k+\left|\int \overline{\varphi}\, d\hat{m}^k-\int\overline{\varphi}\, d\hat{m}\right|.
\end{aligned}
\end{equation}

In order to bound the first term in the last line of Equation (\ref{22}), it is enough to observe that for every $k\in \mathbb{N}$, $$\overline{\varphi}_k(\hat{y},v)\leq\sup|\varphi_k(\hat{x},w)|\leq C,$$ and also that $\overline{\varphi}_k$ converges to $\overline{\varphi}$, because $\overline{h}^k_{\hat{y}}\to \overline{h}_{\hat{y}}$ uniformly. Therefore, given $\varepsilon>0$, there exists $k_0>0$ such that for every $k\geq k_0$, $$\int\left| \overline{\varphi}_k-\overline{\varphi}\right|\, d\hat{m}^k<\frac{\varepsilon}{2}.$$

Finally, in order to bound the second term in Equation (\ref{22}), we observe that $\overline{\varphi}(\hat{y},v)$ is measurable as a function of $\hat{y}$, thus there exist a continuous function $$\psi:\Sig\times\mathbb{P}^1\to\mathbb{R},$$ such that $\sup|\psi|\leq\sup|\varphi|$ and a compact set $K$ with $\mu(\Sig\setminus K)<\frac{\varepsilon}{4 \sup|\varphi|}$ such that $\psi(\hat{y},v)=\overline{\varphi}(\hat{y},v)$ for $(\hat{y},v)\in K\times\mathbb{P}^1$. As $\hat{m}^k\to \hat{m}$ in the weak-$^*$ topology, for $k\in \mathbb{N}$ large enough, we get that $$\left|\int \overline{\varphi}\, d\hat{m}^k-\int\overline{\varphi}\, d\hat{m}\right|\leq\left|\int_K \psi\, d\hat{m}^k-\int_K\psi\, d\hat{m}\right|+\frac{\varepsilon}{4}<\frac{\varepsilon}{2},$$
which concludes the proof of $\overline{m}^k\to\overline{m}$.

Next we show that $\hat{y}\mapsto \overline{m}^k_{\hat{y}}$ is $\mathcal{B}^l$-measurable and as $\overline{m}^k\to\overline{m}$ weakly-$^*$, this implies that $\hat{y}\mapsto\overline{m}_{\hat{y}}$ is also $\mathcal{B}^l$-measurable. 

\begin{claim}
If $\hat{m}^k$ is an $s$-state, then $\hat{y}\mapsto \overline{m}^k_{\hat{y}}$ is $\mathcal{B}^l$-measurable mod $0$ for every $l\geq l_0$.
\end{claim}
\textit{Proof.}
Fix $l\ge l_0$. We need to proof that given any continuous and bounded function $\varphi:\mathbb{P}^1\to\mathbb{R}$, the map $\Phi_k:\Sig\to\R$, defined as $\hat{y}\mapsto \Phi_k(\hat{y})=\int\varphi\, d\overline{m}^k_{\hat{y}}$, is $\mathcal{B}^l$-measurable mod $0$.

Because $\hat{m}^k$ is an $s$-state, the definition of $\overline{m}^k$ guarantees that $$\Phi_k(\hat{y})=\int\varphi\, d\overline{m}^k_{\hat{y}}=\int\varphi\, d\hat{m}^k_{\hat{z}}$$ is constant for  every $\hat{y}\in W^s_{loc}(\hat{z})$, for $\hat{\mu}^u_{\hat{x}_i}$-almost every $\hat{z}\in\mathcal{D}_{i,l}$. Here $\hat{\mu}^u_{\hat{x}_i}$ denotes an element of the disintegration of $\hat{\mu}$ along its local unstable sets. See the definitions before Lemma \ref{lpe}. 

Let $E\subset\R$ be a measurable set. If there exists $\hat{y}\in W^s_{loc}(\hat{z})\cap\Phi_k^{-1}(E)$, then every $\hat{w}\in W^s_{loc}(\hat{z})$ satisfies that $\Phi_k(\hat{w})=\int\varphi\, d\overline{m}^k_{\hat{w}}=\int\varphi\, d\overline{m}^k_{\hat{y}}=\Phi_k(\hat{y})\in E$, concluding that $W^s_{loc}(\hat{z})\subset\Phi_k^{-1}(E)$. 

We have proved that $\Phi_k$ is $\mathcal{B}^l$-measurable, and then $\hat{y}\mapsto \overline{m}^k_{\hat{y}}$ is also $\mathcal{B}^l$-measurable. $\square$ 

\begin{claim}
If $\overline{m}^k\to\overline{m}$ in the weak-$^*$ topology and $\hat{y}\mapsto\overline{m}^k_{\hat{y}}$ is $\mathcal{B}^l$-measurable mod $0$, then $\hat{y}\mapsto\overline{m}_{\hat{y}}$ is $\mathcal{B}^l$-measurable mod $0$ for all $l\geq l_0$.
\end{claim}
\textit{Proof.}
First, we prove that $\hat{y}\mapsto\Phi_k(\hat{y})=\int\varphi(v)\, d\overline{m}^k_{\hat{y}}$ converges in the weak topology of $L^2(\hat{\mu})$ to $\Phi(\hat{y})=\int\varphi\, d\overline{m}_{\hat{y}}$. With that purpose, let $\psi:\Sig\to\R$ be a continuous bounded function  and observe that $$\int \psi\Phi_k\, d\hat{\mu}=\int \psi(\hat{y})\int\varphi(v)\, d\overline{m}^k_{\hat{y}}\, d\hat{\mu}=\int\int \psi(\hat{y})\varphi(v)\, d\overline{m}^k_{\hat{y}}\, d\hat{\mu}.$$

Because $\psi(\hat{y})\varphi(v)$ is also a continuous function and $\overline{m}^k\to \overline{m}$ in the weak-$^{*}$ topology, we deduce that $$\int \psi(\hat{y})\varphi(v)\, d\overline{m}^k\to\int \psi(\hat{y})\varphi(v)\, d\overline{m}=\int \psi\Phi\, d\hat{\mu}.$$ Then, as continuous bounded function are dense in $L^2(\hat{\mu})$, $\Phi_k$ converges weakly to $\Phi$.

By hypothesis, we know that $\Phi_k$ is $\mathcal{B}^l$ measurable mod $0$ for each $l\geq l_0$ and we have to prove that $\Phi$ also is. 

The space $K\subset L^2(\hat{\mu})$ of $\mathcal{B}^l-$measurable functions is convex and close. Then, if $\Phi\not\in K$, by the Hahn-Banach theorem, there exists $\xi\in L^2(\hat{\mu})$ such that $\int \xi \psi\, d\overline{m}=0$ for all $\psi\in K$ and $\int \xi\Phi\, d\hat{\mu}>0$. Since, $0=\int \xi\Phi_k\, d\hat{\mu}  \to \int \xi\Phi\, d\hat{\mu}$, we get a contradiction. 

Finally, we conclude that $\hat{y}\mapsto \Phi(\hat{y})=\int \varphi\, d\overline{m}_{\hat{y}}$ is $\mathcal{B}^l$-measurable mod $0$ for each $l\geq l_0$. $\square$ 

In order to finish the proof of the proposition, we prove that if $\hat{y}\mapsto \overline{m}_{\hat{y}}$ is $\mathcal{B}^l$-measurable for $l\geq l_0$, then $\hat{m}$ is an $s$-state.
  
Since $\hat{y}\to\overline{m}_{\hat{y}}$ is $\mathcal{B}^l$-measurable mod $0$, then for each $l\geq l_0$, there exists a full measure set $E^s_l$ of $\bigcup \{W^s_{loc}(\hat{z}): \ \hat{z}\in\mathcal{D}_{i,l}\}$ that verifies
\begin{equation*}
\hat{z}_1,\hat{z}_2\in E^s_l\cap W^s_{loc}(\hat{z}) \Rightarrow \overline{m}_{\hat{z}_1}=\overline{m}_{\hat{z}_2}, 
\end{equation*}
with $\hat{z}\in\mathcal{D}_{i,l}$, and thus $$(h^{s}_{\hat{z}_1,\hat{z}})_*\hat{m}_{\hat{z}_1}=(h^{s}_{\hat{z}_2,\hat{z}})_*\hat{m}_{\hat{z}_2}\Leftrightarrow (h^{s}_{\hat{z}_1,\hat{z}_2})_*\hat{m}_{\hat{z}_1}=\hat{m}_{\hat{z}_2}.$$
Then, there is a full measure set $E^s=\bigcup E^s_l$, satisfying that if $\hat{y}\in E^s$, then $\hat{y}\in E^s_l$ for some $l$, and as $W^s_{loc}(\hat{y})\subset E^s_l\mod 0$, we conclude that $(h^s_{\hat{z}_1,\hat{z}_2})_*\hat{m}_{\hat{z}_1}=\hat{m}_{\hat{z}_2}$ with $\hat{z}_1,\hat{z}_2\in W^s_{loc}(\hat{y})$ mod $0$. $\square$ 

\section{Characterization of discontinuity points}

We say that a linear cocycle $\hat{A}\in \mathcal{S}_1(\Sig,2)$ is a discontinuity point of the Lyapunov exponents, if there exists a sequence $\{\hat{A}_k\}_{k\in \mathbb{N}}$ such that $\hat{A}_k$ converges to $\hat{A}$ in the Lipschitz topology and $\lambda_{+}(\hat{A}_k)$ does not converges to $\lambda_{+}(\hat{A})$. 

In this section, we use $\mathbb{P}(\hat{A})$-invariant probabilities measures to provide a characterization of these discontinuity points. 

By the semi-continuity of $\lambda_+(\cdot)$ and $\lambda_{-}(\cdot)$, if $\hat{A}$ is a dis\-con\-ti\-nu\-i\-ty point of the Lyapunov exponents, then $\lambda_-(\hat{A})<0<\lambda_+(\hat{A})$. Let $\R^2=E^{s,\hat{A}}_{\hat{x}}\oplus E^{u,\hat{A}}_{\hat{x}}$ be the O\-se\-le\-dets decomposition associated to $\hat{A}$ at the point $\hat{x}\in\Sig$. Consider the measures in $\Sig\times\mathbb{P}^1$ defined by
\begin{equation}\label{msu}
\hat{m}^s=\int \delta_{\mathbb{P}(E^{s,\hat{A}}_{\hat{x}})}\, d\hat{\mu}\quad \text{ and }\quad \hat{m}^u=\int \delta_{\mathbb{P}(E^{u,\hat{A}}_{\hat{x}})}\, d\hat{\mu},
\end{equation}

They are both $\mathbb{P}(\hat{A})$-invariant probability measures projecting to $\hat{\mu}$ and moreover, by Birkhoff's Ergodic Theorem they satisfy, $$\lambda_{-}(\hat{A})= \int \Phi_{\hat{A}}(\hat{x},v)\, d\hat{m}^s\quad \text{ and }\quad \lambda_{+}(\hat{A})= \int \Phi_{\hat{A}}(\hat{x},v)\, d\hat{m}^u,$$
where $\Phi_{\hat{A}}(\hat{x},v)=\log \|\hat{A}(\hat{x})v\|.$ Moreover, if $\hat{A}$ admits invariant holonomies, then $\hat{m}^s$ is an $s$-state and $\hat{m}^u$ is a $u$-state. 

\begin{prop}[Lemma 6.1, \cite{AV}]\label{convex} Let $\hat{A}\in \mathcal{S}_1(\Sig,2)$ be such that $\lambda_{-}(\hat{A})<0<\lambda_{+}(\hat{A})$ and let $\hat{m}$ be a probability measure in $\Sig\times\mathbb{P}^1$ projecting to $\hat{\mu}$. Then $\hat{m}$ is $\mathbb{P}(\hat{A})$-invariant if and only if it is a convex combination of $\hat{m}^s$ and $\hat{m}^u$. \end{prop}

Indeed, one only has to note that every compact subset of $\mathbb{P}^1$ disjoint from $\{\mathbb{P}(E^{s,\hat{A}}), \mathbb{P}(E^{u,\hat{A}})\}$ accumulates on $\mathbb{P}(E^{u,\hat{A}})$ in the future and on $\mathbb{P}(E^{s,\hat{A}})$ in the past. 

The following characterization of discontinuity points was firstly introduce for fiber-bunched cocycles in \cite{AV}. In our statement it is only required for the cocycle to be non-uniformly fiber-bunched. Even thought the proof is the same as in \cite{AV}, we introduce it in here because it exhibits where Proposition \ref{utosdebil} is needed.

\begin{prop}[Proposition 6.3, \cite{AV}]\label{discont} Let $\hat{A}\in \mathcal{S}_1(\Sig,2)$ be a non-uniformly fiber-bunched cocycle. If $\hat{A}$ is a dis\-con\-ti\-nu\-i\-ty point of the Lyapunov exponents, then every $\mathbb{P}(\hat{A})$-invariant probability measure $\hat{m}$ projecting to $\hat{\mu}$ is an $su$-state.\end{prop}
\textit{Proof.}
By the upper semi-continuity of $\lambda_{+}(\cdot)$, passing to a subsequence we may assume $\lim_{k\to\infty}\lambda_{+}(\hat{A}_k)<\lambda_{+}(\hat{A})$. For each $k\in \mathbb{N}$, there exists an ergodic $\mathbb{P}(F_{\hat{A}_k})$-invariant probability measure $\hat{m}^k$ projecting to $\hat{\mu}$ such that
\begin{equation}\label{ident}
\lambda_{+}(\hat{A}_k)=\int \Phi_{\hat{A}_k}(\hat{x},v)\, d\hat{m}^k,
\end{equation}
for $\Phi_{\hat{A}_k}(\hat{x},v)=\log \|\hat{A}_k(\hat{x})v\|.$ If $\lambda_{+}(\hat{A}_k)=0$, then any $\mathbb{P}(\hat{A}_k)$-invariant probability measure $\hat{m}^k$ projecting to $\hat{\mu}$ satisfies Equation (\ref{ident}) and by Theorem \ref{invarianza} is an $su$-state. If $\lambda_{+}(\hat{A}_k)> 0$, we take $\hat{m}^k=\int \delta_{\mathbb{P}(E^{u,k}_{\hat{x}})}\, d\hat{\mu}$, then it satisfies Equation (\ref{ident}) and it is a $u$-state. Consequently, $$\lim_{k\to\infty}\int \Phi_{\hat{A}_k}(\hat{x},v)\, d\hat{m}^k<\lambda_{+}(\hat{A}).$$

Taking sub-sequences again, we may assume that $\hat{m}^k$ converges weak-$^{*}$ to a $\mathbb{P}(F_{\hat{A}})$-invariant probability measure $\hat{m}$. By Proposition \ref{utosdebil}, $\hat{m}$ is a $u$-state and by Proposition \ref{convex}, there exists $\alpha\in [0,1]$ such that $$\hat{m}=\alpha \hat{m}^u+(1-\alpha) \hat{m}^s.$$

Finally, by the uniform convergence of $\Phi_{\hat{A}_k}\to\Phi_{\hat{A}}$, $$\int \Phi_{\hat{A}}(\hat{x},v)\, d\hat{m}=\lim_{k\to\infty}\int \Phi_{\hat{A}_k}(\hat{x},v)\, d\hat{m}^k<\lambda_{+}(A,\hat{\mu})=\int \Phi_{\hat{A}}(\hat{x},v)\, d\hat{m}^u,$$ hence $\hat{m}\neq \hat{m}^u$. It follows that $\alpha\neq 1$ and $$\hat{m}^s=\frac{1}{1-\alpha}\left(\hat{m}-\alpha \hat{m}^u\right)$$ is a $u$-state and therefore an $su$-state. 

Analogously, using $\lambda_{-}(\hat{A})$, we conclude that $\hat{m}^u$ is an $s$-state. Then, Proposition \ref{convex} concludes the statement. $\square$ 

\section{Measures induced by $u$-states}

Recall that $\Sigma^u=\{(x_n)_{n\geq0}: q_{x_nx_{n+1}}=1 \text{ for every } n\geq 0\}$ is the set of sequences with only positive coordinates, $P^u\colon \Sig\to\Sigma^u$ is the projection and for $x\in \Sigma^u$, $W^s_{loc}(x)=(P^u)^{-1}(x)$.

We introduce a type of measures on $\Sigma^u\times\mathbb{P}^1$ that are induced by measures in $\Sig\times\mathbb{P}^1$. For these measures it is possible to identify some geometric properties that are enunciated next.

\begin{defn}\label{induced} A probability measure $m$ on $\Sigma^u\times\mathbb{P}^1$ is induced by a $u$-state if there exist
\begin{enumerate}[label=\emph{(\roman*)}]
\item  a continuous linear cocycle $\hat{A}:\Sig\to SL(2,\R)$ that is constant along local stable sets and admits a unstable holonomy $H^u$,
\item and a $\mathbb{P}(\hat{A})$-invariant probability measure $\hat{m}$ on $\Sig\times\mathbb{P}^1$ projecting to $\hat{\mu}$ such that $\hat{m}$ is a $u$-state for $h^u=\mathbb{P}(H^u)$ and $m=(P^u\times Id)_*\hat{m}$.
\end{enumerate}
\end{defn}

Note that $m=(P^u\times Id)_*\hat{m}$ is a $\mathbb{P}(A)$-invariant measure where $A\colon \Sigma^u\to SL(2,\R)$ is such that $\hat{A}=A\circ P^u$. 

If $\{\hat{m}_{\hat{x}}\}_{\hat{x}\in\Sig}$ is a disintegration of $\hat{m}$ along the fibers and $\{\hat{\mu}_x\}_{x\in\Sigma^u}$ is a disintegration of $\hat{\mu}$ as in Lemma \ref{lpe}, then for $x\in\Sigma^u$
\begin{equation*}\label{projectm}
m_x=\int_{W^s_{loc}(x)}\hat{m}_{\hat{x}}d\hat{\mu}_x(\hat{x})
\end{equation*}
is a disintegration of $m$ along the fibers of $\Sigma^u\times\mathbb{P}^1$.

We remark that the unstable holonomy in item (i) of Definition \ref{induced} can be either uniform as in Definition \ref{holonomiafuerte} or non-uniform as in Definition \ref{holonomiadebil}. The first case has been studied in Section 4.3 of \cite{BBB}. In the following we focus on the second case. Notice that we are only asking for the holonomy to satisfy Definition \ref{holonomiadebil}, we do not required for the cocycle to be non-uniformly fiber-bunched neither the holonomy to be given by Proposition \ref{existdebil}. 

\begin{prop}\label{unou}
Any probability measure $m$ induced by a $u$-state admits a disintegration into conditional measures $\{m_x\}_{x\in\Sigma^u}$ that are defined for every $x\in\Sigma^u$ and vary continuously with $x$ in the weak-$^*$ topology.
\end{prop}
\textit{Proof.}
Let $l$ be large enough such that the unstable holonomy block $D^u_l$ in Definition \ref{holonomiadebil} satisfies $\mathcal{D}^u_{l}\cap[0;i]\neq\emptyset$ for every $i\in\{1,\ldots,d\}$. 

For every $i\in\{1,\ldots,d\}$, fix $\hat{x}_i\in[0;i]$ and define $$\mathcal{D}_{i,l}=\{\hat{z}\in W^s_{loc}(\hat{x}_i): \ W^u_{loc}(\hat{z})\cap\mathcal{D}^u_{l}\neq\emptyset\},$$
and $$\mathcal{C}_{i,l}=\bigcup\{W^u_{loc}(\hat{z}): \ \hat{z}\in\mathcal{D}_{i,l}\}.$$ Observe that the sets $C_{i,l}$ do not depend of the choice of $\hat{x}_i$ and $\mathcal{D}^u_{l}\cap[0;i]\subset\mathcal{C}_{i,l}$.

In the following, $\hat{\mu}_{x}$ denotes the disintegration of $\hat{\mu}$ along $W^s_{loc}(x)$, $x\in \Sigma^u$, as in Lemma \ref{lpe}.

Let $\varphi:\mathbb{P}^1\to\R$ be a continuous function and consider $x,y\in\Sigma^u$ in the same cylinder $[0;i]$, then 
\begin{equation*}
\begin{aligned}
&\int_{\mathbb{P}^1}\varphi(v)dm_y=\int_{W^s_{loc}(y)}\int_{\mathbb{P}^1}\varphi d\hat{m}_{\hat{y}}d\hat{\mu}_y(\hat{y})\\
=&\int_{W^s_{loc}(x)\cap\mathcal{C}_{i,l}}\left(\int_{\mathbb{P}^1}\varphi\circ h^u_{\hat{x},\hat{y}} d\hat{m}_{\hat{x}}\right)R_{x,y}(\hat{x})d\hat{\mu}_x(\hat{x})+\int_{W^s_{loc}(y)\cap\mathcal{C}^c_{i,l}}\int_{\mathbb{P}^1}\varphi d\hat{m}_{\hat{y}}d\hat{\mu}_y(\hat{y}).\\
\end{aligned}
\end{equation*}
As $$\hat{\mu}_y({W^s_{loc}(y)\cap \mathcal{C}^c_{i,l}})=\int_{W^s_{loc}(y)\cap \mathcal{C}^c_{i,l}}\psi(\hat{y})d\mu^s=\hat{\mu}(\mathcal{C}^c_{i,l})<\frac{1}{l},$$we deduce that,
\begin{equation*}
\begin{aligned}
\left|\int_{\mathbb{P}^1}\varphi dm_y-\int_{\mathbb{P}^1}\varphi dm_x\right|=&\int_{W^s_{loc}(x)\cap\mathcal{C}_{i,l}}\int_{\mathbb{P}^1}\left| \varphi\circ h^u_{\hat{x},\hat{y}}\cdot R_{x,y}(\hat{x})-\varphi\right| d\hat{m}_{\hat{x}}d\hat{\mu}_x(\hat{x})\\
+&\int_{W^s_{loc}(x)\cap \mathcal{C}^c_{i,l}}\int_{\mathbb{P}^1}|\varphi|d\hat{m}_{\hat{x}}d\hat{\mu}_x(\hat{x})\\
+&\int_{W^s_{loc}(y)\cap \mathcal{C}^c_{i,l}}\int_{\mathbb{P}^1}|\varphi|d\hat{m}_{\hat{y}}d\hat{\mu}_y(\hat{y}).\\
\leq&\int_{W^s_{loc}(x)\cap\mathcal{C}_{i,l}}\int_{\mathbb{P}^1}\left| \varphi\circ h^u_{\hat{x},\hat{y}}\cdot R_{x,y}(\hat{x})-\varphi\right| d\hat{m}_{\hat{x}}d\hat{\mu}_x(\hat{x})+\frac{2C}{l},
\end{aligned}
\end{equation*}
where $C=\sup|\varphi|$. 

By item (c) in Definition \ref{holonomiadebil}, we know that $\|h^u_{\hat{x},\hat{y}}-Id\|$ is uniformly small when $\hat{x},\hat{y}$ are close, $\hat{x},\hat{y}\in W^u_{loc}(\hat{z})$ and $\hat{z}\in\mathcal{C}_{i,l}$. Then, let $\varepsilon>0$ and set $l$ such that $\frac{4\sup|\varphi|}{\varepsilon}<l$, by Lemma \ref{lpe}, we can choose $\delta>0$, such that if $d(x,y)<\delta$, then the expression $\|R_{x,y}-1\|_{L^1}$ is small and thus $\|\varphi\circ h^u_{x,y}\cdot R_{x,y}(\hat{x})-\varphi\|<\frac{\varepsilon}{2}$, concluding that $\left|\int \varphi dm_x-\int \varphi dm_y\right|<\varepsilon$. $\square$ 

We want to apply this proposition to cocycles that satisfy the hypotheses of Theorem A. Therefore, let $\hat{A}\in \mathcal{S}_{1}(\Sig,2)$ be a non-uniformly fiber-bunched cocycle which admits a uniform stable holonomy. 

The following construction is due to Corollary 1.15 of \cite{BGV}.

Fix $d$ points $\hat{z}_1,\ldots\hat{z}_d$ such that $\hat{z}_i\in[0;i]$ for every $i\in \{1,...,d\}$. For each $\hat{x}\in [0;i]$, let $g(\hat{x})$ be the unique point in the intersection $W^u_{loc}(\hat{z}_i)\cap W^s_{loc}(\hat{x})$. Note that $g(\hat{x})=g(\hat{y})$ if $\hat{y}\in  W^s_{loc}(\hat{x})$. Define
\begin{equation}\label{b2}
\widetilde{A}(\hat{x})=H^{s}_{f(\hat{x}),g(f(\hat{x}))}\circ \hat{A}(\hat{x})\circ H^{s}_{g(\hat{x}),\hat{x}}.
\end{equation}

By Equation (\ref{b2}) and item (b) in Definition \ref{holonomiafuerte}, we conclude that $\widetilde{A}$ is constant along local stable sets.  
 
Since $\hat{A}$ is a non-uniformly fiber-bunched cocycle, Proposition \ref{existdebil} implies the existence of a non-uniform unstable holonomy $H^u$ for $\hat{A}$. 

Let $\widetilde{A}$ be the cocycle defined by Equation (\ref{b2}). We claim that $\widetilde{A}$ admits a non-uniform unstable holonomy $\widetilde{H}^u$. In order to prove this, we consider the $\hat{\mu}$-full measure set $M^u$ in Definition \ref{holonomiadebil} and $\hat{y}, \hat{z}\in W^u_{loc}(\hat{x})$ with $\hat{x}\in M^u$, then define $$\widetilde{H}^u_{\hat{y},\hat{z}}=H^s_{\hat{z},g(\hat{z})}\circ H^u_{\hat{y},\hat{z}} \circ H^s_{g(\hat{y}),\hat{y}},$$ where $H^s$ denotes the uniform stable holonomy of $\hat{A}$ and $g(\cdot)$ has been defined above. Notice that $\widetilde{H}^u$ verifies item (a)-(c) in Definition \ref{holonomiadebil}.

Given $\hat{m}$ a $\mathbb{P}(\hat{A})$-invariant probability measure projecting to $\hat{\mu}$, we construct a new measure $\widetilde{m}$ which is a $\mathbb{P}(\widetilde{A})$-invariant probability measure also projecting to $\hat{\mu}$. Let $\{\hat{m}_{\hat{x}}\}$ be a disintegration of $\hat{m}$, define,
\begin{equation}\label{tutia2}
\widetilde{m}_{\hat{x}}= \left(h^{s}_{\hat{x},g(\hat{x})}\right)_*\hat{m}_{\hat{x}} \quad \text{ and } \quad \widetilde{m}=\int \widetilde{m}_{\hat{x}}\; d\hat{\mu}.
\end{equation} 
Here $h^s=\mathbb{P}(H^s)$.

Observe that if $\hat{m}$ is a $u$-state for $(\hat{A}, H^u)$, then $\widetilde{m}$ is a $u$-state for $(\widetilde{A},\widetilde{H}^u)$. 

If $(\hat{A}_k, H^{s,k})\to (\hat{A},H^s)$ in $\mathcal{H}^s$ and $\widetilde{A}_k$ denotes the cocycle defined by Equation (\ref{b2}) applied to $\hat{A}_k$ and $H^{s,k}$, then $\widetilde{A}_k\to \widetilde{A}$ in the $C^0$ topology and $\widetilde{A}_k$ admits a non-uniform unstable holonomy $\widetilde{H}^{u,k}$ for every $k\in \mathbb{N}$. By Proposition \ref{holonomia}, we know that fixed $l\in \mathbb{N}$, there exists $k_l\in \mathbb{N}$ such that the continuity in item (c) of Definition \ref{holonomiadebil} can be taken to be uniform for $\widetilde{H}^u$ and $\widetilde{H}^{u,k}$ for every $k\geq k_l$. 

Let $\hat{m}^k$ be $u$-states for $\hat{A}_k$ such that $\hat{m}^k\to \hat{m}$ in the weak-$^{*}$ topology. Then, by Proposition \ref{utosdebil}, $\hat{m}$ is a $u$-state for $\hat{A}$. Define $\widetilde{m}^k$ and $\widetilde{m}$ by Equation (\ref{tutia2}) applied to $h^s$ and $h^{s,k}$ respectively. 

Consider $m^k=(P^u\times Id)_*\widetilde{m}^k$ and $m=(P^u\times Id)_*\widetilde{m}$. Observe that $m^k$ and $m$ are measures induced by $u$-states as in Definition \ref{induced}. Therefore, Proposition \ref{unou} holds for $m$ and for every $m^k$, $k\in \mathbb{N}$. Moreover, the observation above about the equicontinuity of the holonomies $\widetilde{H}^{u,k}$ allow us to conclude the following result.

\begin{prop}\label{47} The measures $m^k$ and $m$ admit disintegrations $\{m^k_x\}_{x\in\Sigma^u}$ and $\{m_x\}_{x\in\Sigma^u}$, respectively, which are defined for every $x\in\Sigma^u$ and such that for every continuous function $\varphi:\mathbb{P}^1\to\R$ and $\varepsilon>0$ there exists $\delta>0$ and $k_l\in \mathbb{N}$ such that $d(x,y)<\delta$ implies $|\int \varphi dm_x-\int \varphi dm_y|<\varepsilon$ and $|\int \varphi dm^k_x-\int \varphi dm^k_y|<\varepsilon$ for every $k\geq k_l$.\end{prop}

The next proposition is a consequence of Proposition \ref{47} and the fact that $\widetilde{m}^k\to \widetilde{m}$ in the weak-$^{*}$ topology. The proof is analogous to the proof of Proposition 4.8 of \cite{BBB}. 

\begin{prop}\label{uniform} Let $\hat{A}\in \mathcal{S}_{1}(\Sig,2)$ be a non-uniformly fiber-bunched cocycle which admits a uniform stable holonomy. Assume $(\hat{A}_k, H^{s,k})\to (\hat{A},H^s)$ in $\mathcal{H}^s$ and $\hat{m}^k$ are $u$-states for $\hat{A}$ such that $\hat{m}^k \to \hat{m}$ in the weak-$^{*}$ topology.

If $\widetilde{m}^k$ and $ \widetilde{m}$ has been defined by Equation (\ref{tutia2}), $m^k=(P^u\times Id)_*\widetilde{m}^k$, $m=(P^u\times Id)_*\widetilde{m}$ and $\{m^k_x\}$ and $\{m_x\}$ denote the disintegrations of $m^k$ and $m$ given by Proposition \ref{47}, respectively, then $m_x^k\to m_x$ uniformly on x.\end{prop}

\section{Proof of the Theorems}
\subsection{Proof of Theorem A.}

Recall that $\hat{f}$ is a sub-shift of finite type and $\hat{\mu}$ is an ergodic $\hat{f}$-invariant probability measure with local product structure and fully supported. 

\begin{theoA} Let $\hat{A}\in \mathcal{S}_{1}(\Sig,2)$ be a non-uniformly fiber-bunched cocycle which admits a uniform stable holonomy.  If $(\hat{A}_k, H^{s,k})\to (\hat{A},H^s)$ in $\mathcal{H}^s$, then $\lambda_{+}(\hat{A}_k)\to \lambda_{+}(\hat{A})$. 
\end{theoA}

Applying the results obtained in Section 7, we are able to extend the argument of \cite{BBB} to conclude Theorem A. However, we use this approach only on the second part of the proof. For the first part, more precisely for Case I below, we use a new strategy which we consider more efficient since shows clearly which hypotheses are absolute necessary and which ones can be weakened. For example, we remark that it is enough to guarantee the existence of the non-uniform unstable holonomy only in finite set of points. This observation can be useful in order to prove the general conjecture of Viana. 

\emph{Proof of Theorem A.} We prove the result by contradiction. That is, suppose that the sequence $(\hat{A}_k,H^{s,k})$ verifies $\hat{A}_k\to \hat{A}$ in the Lipschitz topology and $H^{s,k}$ converges uniformly to $H^s$, but $\lambda_+(\hat{A}_k)$ does not converges to $\lambda_+(\hat{A})$. In particular, this implies $\lambda_-(\hat{A})<0<\lambda_+(\hat{A})$. 

For every $k\in \mathbb{N}$, we denote by $\hat{m}^k$ the ergodic measure that verifies \begin{equation}\label{fust} \lambda_{+}(\hat{A}_k)=\int \Phi_{\hat{A}_k}(\hat{x},v)\, d\hat{m}^k\quad \text{for} \quad \Phi_{\hat{A}_k}(\hat{x},v)=\log \|\hat{A}_k(\hat{x})v\|.\end{equation} We refer the reader to the proof of Proposition \ref{discont} for an argument that implies the existence of these measures. 

Proposition \ref{yanise} gives us a property about su-states which allows us to understand better the nature of  these measures. Thus, the proof of Theorem A is divided into two cases. First, we study the case where there exists a subsequence $j_k$ such that $\hat{m}^{j_k}$ is an $su$-state for every $k\in \mathbb{N}$. For the second case we assume that there exists $k_0\in \mathbb{N}$ such that $\hat{m}^{k}$ is not an $su$-state for every $k\geq k_0$. In order to simplify the notation in the first case, we denote $\hat{m}^{j_k}$ as $\hat{m}^k$. 

\subsection*{Case I: $\hat{m}^k$ are $su$-states.}

By Kalinin \cite{K}, we know that there exists a periodic point $\hat{p}$ of $\hat{f}$ such that $\hat{A}^{n_p}(\hat{p})$ is hyperbolic, where $n_p=per(\hat{p})$. 

Let $i_{\hat{p}}$ be the element in $\{1,...,d\}$ such that $\hat{p}\in [0;i_{\hat{p}}]$ and let $a=\mathbb{P}(E_1)$ and $r=\mathbb{P}(E_2)$, where $E_1$ and $E_2$ are the subspaces of $\mathbb{R}^2$ associated to the eigenvalues of $\hat{A}^{n_p}(\hat{p})$.

Let $\hat{m}^u$ and $\hat{m}^s$ be the measures defined by Equation (\ref{msu}). That is, $\hat{m}^u$ and $\hat{m}^s$ are supported on the Oseledets subspaces associated to $\lambda_{+}(\hat{A})$ and $\lambda_{-}(\hat{A})$ respectively. By Proposition \ref{yanise} and Proposition \ref{discont}, they admit continuous disintegrations satisfying that $\hat{m}^u_{\hat{z}}=\delta_{a_{\hat{z}}}$ and $\hat{m}^s_{\hat{z}}=\delta_{r_{\hat{z}}}$, which implies that the maps $\hat{z}\mapsto a_{\hat{z}}$ and $\hat{z}\mapsto r_{\hat{z}}$ are continuous. In particular, $\hat{m}^u_{\hat{p}}=\delta_a$ and $\hat{m}^s_{\hat{p}}=\delta_{r}$.  Therefore, the sets $$M^+=\{(\hat{z},a_{\hat{z}})\}_{\hat{z}\in\Sig}\, \text{ and }\, M^-=\{(\hat{z},r_{\hat{z}})\}_{\hat{z}\in\Sig}$$ are compact and disjoint. In particular, there exists an $\varepsilon>0$ such that 
\begin{equation}\label{distanceM}
d(M^-,M^+)>\varepsilon.
\end{equation} 

Since hyperbolicity is an open condition and $\hat{A}_k$ converges to $\hat{A}$, for $k$ large enough, $\hat{A}_k^{n_p}(\hat{p})$ is also hyperbolic. We denote as $\{a_k,r_k\}$ the elements of $\mathbb{P}^1$ defined by the subspaces of $\mathbb{R}^2$ associated to the eigenvalues of $\hat{A}_k^{n_p}(\hat{p})$. Then, we conclude that $a_k\to a$ and $r_k\to r$.

Applying Proposition \ref{yanise} to the measures $\hat{m}^k$, we deduce that each $\hat{m}^k$ admits a disintegration $\{\hat{m}_{\hat{z}}^k\}_{\hat{z}\in\Sig}$ such that $\hat{z}\mapsto \hat{m}_{\hat{z}}^k$ is continuous and invariant by the holonomies. 

As a consequence of the ergodicity of $\hat{\mu}$, we know that for each $i\in \{1,...,d\}\setminus \{i_{\hat{p}}\}$ there exists $j_i>0$ such that $\hat{f}^{j_i}([0;i_{\hat{p}}])\cap[0;i]$ is a positive measure set. For every $i\in \{1,...,d\}\setminus \{i_{\hat{p}}\}$ define $j_i$ as the smaller integer with this property and for $i=i_{\hat{p}}$ consider $j_i=0$.

Let $l$ be large enough such that the unstable holonomy block of $\hat{A}$ in Definition \ref{block} satisfies $\hat{f}^{j_i}([0;i_{\hat{p}}])\cap [0;i]\cap\mathcal{D}^u_{\hat{A},l}\neq\emptyset$ for every $i\in \{1,...,d\}$ and fix $\hat{x}_i\in \hat{f}^{-j_i}([0;i]\cap\mathcal{D}^u_{\hat{A},l})\cap [0;i_{\hat{p}}]$.

For each $\hat{z}\in [0;i_{\hat{p}}]$ define, $$a^k_{\hat{z}}=h^k_{\hat{z}}a_k,$$ where  $$h^k_{\hat{z}}= h^{s,k}_{\hat{y_2},\hat{z}}\circ h^{u,k}_{\hat{y_1},\hat{y_2}}\circ h^{s,k}_{\hat{p},\hat{y_1}},$$ and $$\hat{y_1}\in W^s_{loc}(\hat{p})\cap W^u_{loc}(\hat{x}_{i_{\hat{p}}}) \text{ and } \hat{y_2}\in  W^s_{loc}(\hat{z})\cap W^u_{loc}(\hat{x}_{i_{\hat{p}}}).$$

If $\hat{z}\in [0;i]$ with $i\neq i_{{\hat{p}}}$, we define $a^k_{\hat{z}}$ as follows $$a^k_{\hat{z}}=h^{su,k}_{\hat{f}^{j_i}(\hat{x}_i),\hat{z}}\circ \mathbb{P}(\hat{A}^{j_i}_k(\hat{x}_i))a^k_{\hat{x}_i},$$ where  $$h^{su,k}_{\hat{f}^{j_i}(\hat{x}_i),\hat{z}}=h^{s,k}_{\hat{y},\hat{z}}\circ h^{u,k}_{\hat{f}^{j_i}(\hat{x}_i),\hat{y}}$$ and $$y\in W^u_{loc}(\hat{f}^{j_i}(\hat{x}_i))\cap W^s_{loc}(\hat{z}).$$

We construct the map $\hat{z}\mapsto r^k_{\hat{z}}$ analogously using $r_k$ instead of $a_k$.  

The maps $\hat{z}\mapsto a^k_{\hat{z}}$ and $\hat{z}\mapsto r^k_{\hat{z}}$ are continuous. This is due to the continuity of $h^k_{\hat{z}}$, $\hat{A}_k$ and $h^{su,k}$. Also, the topology in $\mathcal{H}^s$ and Proposition \ref{holonomia} allow us to prove that $a^k_{\hat{z}}\to a_{\hat{z}}$ and $r^k_{\hat{z}}\to r_{\hat{z}}$ uniformly on $\Sig$. Recall that $\hat{z}\mapsto a_{\hat{z}}$ and $\hat{z}\mapsto r_{\hat{z}}$ are defined by the disintegrations of $\hat{m}^u$ and $\hat{m}^s$ respectively. 

Since $\hat{m}^k$ is invariant by $\mathbb{P}(\hat{A}_k)$ and it admits a continuous disintegration $\{\hat{m}^k_{\hat{z}}\}$, we infer that
\begin{equation}\label{invm}
\mathbb{P}(\hat{A}_k(\hat{z}))_*\hat{m}_{\hat{z}}^k=\hat{m}_{\hat{f}(\hat{z})}^k\quad \text{ for every } \hat{z}\in \hat{\Sigma}.
\end{equation}
In particular, $\text{supp }\hat{m}^k_{\hat{p}}\subset \{a_k,r_k\}$. Moreover, $\{\hat{m}^k_{\hat{z}}\}$ is invariant by holonomies, then for every $\hat{z}\in \hat{\Sig}$, $\#\text{supp } \hat{m}^k_{\hat{z}}=\#\text{supp } \hat{m}^k_{\hat{p}}$ and $\text{supp } \hat{m}^k_{\hat{z}}\subset \{a^k_{\hat{z}},r^k_{\hat{z}}\}$ .

First, we consider the case $\#\text{supp } \hat{m}^k_{\hat{p}}=2$, that is, $\text{supp } \hat{m}^k_{\hat{z}}=\{a^k_{\hat{z}},r^k_{\hat{z}}\}$ for every $\hat{z}\in \hat{\Sig}$. We suppose that $\hat{z}\mapsto a^k_{\hat{z}}$ and $\hat{z}\mapsto r^k_{\hat{z}}$ are not $\mathbb{P}(\hat{A}_k)$-invariant sections. 

By Equation (\ref{invm}), $$\mathbb{P}(\hat{A}_k(\hat{z}))(\{a_{\hat{z}}^k, r_{\hat{z}}^k\})= \{a_{\hat{f}(\hat{z})}^k, r_{\hat{f}(\hat{z})}^k\}$$ for each $\hat{z}\in\Sig$. Since we assume that the sections are not invariant, we deduce that for every $k\in \mathbb{N}$ there exist $j_k>k$ and $\hat{z}_{j_k}$ such that $$\mathbb{P}(\hat{A}_{j_k}(\hat{z}_{j_k}))(a^{j_k}_{\hat{z_{j_k}}})=r^{j_k}_{\hat{f}(\hat{z}_{j_k})}.$$

By compactness of $\Sig$, there exists $\hat{z}_0$ where the sequence $\{\hat{z}_{j_k}\}$ accumulates. In order to simplify the notation we suppose that $j_k=k$. 

Next, taking $\varepsilon>0$ as in Equation (\ref{distanceM}), we choose $k$ large enough such that $$d(a^k_{ \hat{z}_0 },a_{ \hat{z}_0 })<\frac{\varepsilon}{6}\quad \text{ and }\quad d(r^k_{ \hat{z}_0 },r_{ \hat{z}_0 })<\frac{\varepsilon}{6}.$$ Therefore, 
\begin{equation*}\|r_{\hat{f}( \hat{z}_0 )}-a_{\hat{f}( \hat{z}_0 )}\|\leq \|r_{\hat{f}( \hat{z}_0 )}-r^k_{\hat{f}(\hat{z}_k)}\|+\|\mathbb{P}(\hat{A}_k(\hat{z}_k))(a^k_{\hat{z}_k})- \mathbb{P}(\hat{A}( \hat{z}_0 ))(a_{ \hat{z}_0 })\|<\frac{\varepsilon}{3},
\end{equation*}
which is a contradiction to the fact that $M^+$ and $M^-$ are separated sets. 

Therefore, there exists a $k_0\in \mathbb{N}$ such that for every $k\geq k_0$, $\hat{z}\mapsto a^k_{\hat{z}}$ and $\hat{z}\mapsto r^k_{\hat{z}}$ are $\mathbb{P}(\hat{A}_k)$-invariant sections. Then, it is possible to define two $\mathbb{P}(\hat{A}_k)$-invariant measures, $$\hat{m}^{s,k}=\int\delta_{r^k_{\hat{z}}}\, d\hat{\mu} \quad \text{ and }\quad \hat{m}^{u,k}=\int\delta_{a^k_{\hat{z}}}\, d\hat{\mu}.$$ Moreover, we deduce that $\hat{m}^k=\alpha_k \hat{m}^{s,k}+(1-\alpha_k)\hat{m}^{u,k}$ with $\alpha_k\neq0$, which contradicts the ergodicity of $\hat{m}^k$. This contradiction arises from the assumption of $\#\text{supp }\hat{m}_{\hat{p}}^k=2$. Summarizing, we know that $\#\text{supp }\hat{m}_{\hat{p}}^k=1$ and thus either $\hat{m}^k=\hat{m}^{u,k}$ or $\hat{m}^k=\hat{m}^{s,k}$. 

We consider the case $\hat{m}^k=\hat{m}^{u,k}$ and conclude that $\lambda_+(\hat{A}_k)$ converges to $\lambda_+(\hat{A})$. This is not possible since $\hat{A}$ is a discontinuity point of the Lyapunov exponents, therefore arriving at a contradiction. 

In the following, we prove that $\hat{m}^{k}\to \hat{m}^u$ in the weak-$^*$ topology. Take any continuous bounded function $\varphi:\Sig\times\mathbb{P}^1\to\R$ and define $\psi_k(\hat{z}):=\varphi(\hat{z},a^k_{\hat{z}})$ for every $k\in \mathbb{N}$. Analogously, define $\psi$ using $a_{\hat{z}}$ instead of $a^k_{\hat{z}}$. Thus, these continuous functions satisfy that for each $\hat{z}\in\Sig$, $$\mid \psi_k(\hat{z})\mid<\sup_{(\hat{x},v)\in\Sig\times\mathbb{P}^1}\mid\varphi(\hat{x},v)\mid.$$ As $a_{\hat{z}}^k\to a_{\hat{z}}$ uniformly in $\Sig$, we get that $\psi_k(\hat{z})\to \psi(\hat{z})$, then the dominated convergence theorem implies that $$\int \varphi\, d\hat{m}^k=\int\psi_k\, d\hat{\mu}\to\int\psi\, d\hat{\mu}=\int \varphi\, d\hat{m}^u.$$

Finally, using Equation (\ref{fust}) and the definition of $\hat{m}^u$, we infer that $$\lambda_+(\hat{A}_k)=\int \Phi_{\hat{A}_k}(\hat{x},v)\, d\hat{m}^k\to\int \Phi_{\hat{A}}(\hat{x},v)\, d\hat{m}^u=\lambda_+(\hat{A}),$$
which is a contradiction. 

If $\hat{m}^k=\hat{m}^{s,k}$, then the same argument as above shows that $\lambda_+(\hat{A}_k)=0=\lambda_-(\hat{A}_k)$ and it converges to $\lambda_-(\hat{A})$. This concludes that $\lambda_+(\hat{A})=0=\lambda_-(\hat{A})$ which again is not possible.

Therefore, we have established that Case I is not compatible with $\hat{A}$ being a discontinuity point of the Lyapunov exponents.

\subsection*{Case II: $\hat{m}^k$ are not $su$-states for every $k\geq k_0$.}

This second case is divided in two parts. 

\subsubsection*{Part I.} We want to apply the results of Section 7 to this context. 

Let $\hat{A}$ and $\hat{A}_k$ be the cocycles in the statement of Theorem A. Recall that for every $k\in \mathbb{N}$, $\hat{m}^k$ is an ergodic measure that verifies Equation (\ref{fust}). We can assume without loss of generality that $\hat{m}^k$ converges in the weak-$^{*}$ topology to some measure $\hat{m}$. Observe that $\hat{m}$ is a $\mathbb{P}(\hat{A})$-invariant probability measure projecting to $\hat{\mu}$ and by Proposition \ref{convex}, there exists $\alpha\neq 1$ such that $\hat{m}=\alpha\hat{m}^u+(1-\alpha)\hat{m}^s$.

Let $\widetilde{A}$ and $\widetilde{A}_k$ defined by Equation (\ref{b2}) applied to $(\hat{A}, H^s)$ and $(\hat{A}^k, H^{s,k})$ respectively.

For every $k\in \mathbb{N}$, we consider the measure $\widetilde{m}^k$ defined by Equation (\ref{tutia2}) applied to a disintegration of $\hat{m}^k$ and $h^{s,k}=\mathbb{P}(H^{s,k})$. Analogously, we define $\widetilde{m}$.  

Recall that $P^u\colon \hat{\Sigma}\to \Sigma^u$, $W^s_{loc}(x)=(P^u)^{-1}(x)$ for every $x\in \Sigma^u$ and $\hat{\mu}_x$ denotes the disintegration of $\hat{\mu}$ along $W^s_{loc}(x)$ as in Lemma \ref{lpe}.

Define $m^k=(P^u\times Id)_*\widetilde{m}^k$ and $m=(P^u\times Id)_*\widetilde{m}$. Observe that $m^k$ and $m$ are measures induced by $u$-states as in Definition \ref{induced}. Moreover, Proposition \ref{47} and Proposition \ref{uniform} hold. 

In the following, for every $k\geq k_0$, $\{m^k_x\}$ denotes the disintegration of $m^k$ given by Proposition \ref{47}.

\begin{prop}\label{nonatomic} The measures $m^k_x$ are non-atomic for every $x\in \Sigma^u$ and for every $k\geq k_0$.
\end{prop}

In order to prove the proposition above, we need the next two lemmas from \cite{BV04}. 

\begin{lemma}[Lemma 5.2, \cite{BV04}]\label{ee} If there exists $x\in \Sigma^u$ such that $m_x^k$ is atomic, then there exists $\gamma_k>0$ such that for every $y\in \Sigma^u$, there exists $v_y^k\in\mathbb{P}^1$ such that the measure $m^k_y$ satisfies $\gamma_k=m_y^k(v_y^k)>0$.
\end{lemma} 

The lemma above implies that if one element of the disintegration has an atom, then every other element of the disintegration also has an atom. Observe that $\gamma_k$ does not depend on $y$.

\begin{lemma}[Lemma 5.3, \cite{BV04}]\label{ews} If $y\in\Sigma^u$ verifies $m_y^k(v_y^k)>0$ for some $v_y^k\in\mathbb{P}^1$, then for $\hat{\mu}_{y}$-almost every $\hat{y}\in W^s_{loc}(y)$, $\widetilde{m}_{\hat{y}}^k(v^k_y)>0$. Moreover, $m_y^k(v_y^k)=\widetilde{m}_{\hat{y}}^k(v^k_y).$ 
\end{lemma} 

\subsubsection*{Proof of Proposition \ref{nonatomic}.} Suppose, by contradiction, that there exist $k\geq k_0$ and $x\in \Sigma^u$ such that $m_x^k$ is an atomic measure.

By Theorem \ref{invarianza} and the assumption of $\hat{m}^k$ not being a $su$-state, we deduce that $\lambda_+(\hat{A}_k)>0$. Therefore, $$\hat{m}^k=\int\delta_{\mathbb{P}(E^{u,k}_{\hat{x}})}d\hat{\mu}.$$ 

Since $\widetilde{m}^k$ is defined by Equation (\ref{tutia2}), any disintegration of $\widetilde{m}$, $\{\widetilde{m}^k_{\hat{x}}\}$, verifies that $\widetilde{m}^k_{\hat{x}}$ has only one atom for almost every $\hat{x}\in \hat{\Sigma}$. This last observation combined with Lemma \ref{ee} and Lemma \ref{ews} shows that $\gamma_k$ is equal to 1 and there exists a full measure subset $E$ of $\hat{\Sigma}$, such that $\widetilde{m}_{\hat{x}}=\widetilde{m}_{\hat{y}}$ for every $\hat{x},\hat{y}\in E$ such that $P^u(\hat{x})=P^u(\hat{y})$. 

Let $\hat{x},\hat{y}\in E$ such that $\hat{y}\in W^s_{loc}(\hat{x})$ and define $g(\cdot)$ as in Section 7. Therefore, by Equation (\ref{tutia2}), $$\left(h^{s}_{\hat{y},g(\hat{y})}\right)_*\hat{m}_{\hat{y}}=\widetilde{m}_{\hat{y}}=\widetilde{m}_{\hat{x}}=\left(h^{s}_{\hat{x},g(\hat{x})}\right)_*\hat{m}_{\hat{x}}.$$
Using that $g(\hat{x})=g(\hat{y})$ if $\hat{y}\in W^s_{loc}(\hat{x})$, we conclude that $$\hat{m}_{\hat{y}}=\left(h^{s}_{g(\hat{x}),\hat{y}}\circ h^{s}_{\hat{x},g(\hat{x})}\right)_*\hat{m}_{\hat{y}}=(h^{s}_{\hat{x},\hat{y}})_*\hat{m}_{\hat{x}}.$$

This last equation implies that $\hat{m}^k$ is an $s$-state, and therefore an $su$-state which contradicts the assumption in Case II. This contradiction arises from the assumption that there exist $k\geq k_0$ and $x\in \Sigma^u$ such that $m^k_x$ is an atomic measure, then the proposition follows. $\square$

Summarizing, we suppose that $(\hat{A}_k,H^{s,k})\to (\hat{A},H^s)$ in $\mathcal{H}^s$, but $\lambda_+(\hat{A}_k)$ does not converges to $\lambda_+(\hat{A})$. Moreover, we assume that there exist $k_0\in \mathbb{N}$ such that the ergodic measures $\hat{m}^k$ that satisfy Equation (\ref{fust}) are not $su$-states for every $k\geq k_0$ and $\hat{m}^k\to \hat{m}$ in the weak-$^{*}$ topology. By Proposition \ref{convex}, $\hat{m}=\alpha\hat{m}^u+(1-\alpha)\hat{m}^s$ with $\alpha\neq 1$ and by Proposition \ref{discont}, $\hat{m}$ is an $su$-state. 

Applying Equation (\ref{b2}) to $(\hat{A},H^s)$ and $(\hat{A}_k, H^{s,k})$ for every $k\in \mathbb{N}$, we can define linear cocycles $\widetilde{A},\widetilde{A}_k\colon \Sigma\to SL(2,\mathbb{R})$ which are constant along local stable sets. This implies that there exist cocycles $A, A_k\colon \Sigma^u\to SL(2,\mathbb{R})$ such that $\widetilde{A}=A\circ P^u$, $\widetilde{A}_k=A_k\circ P^u$ and $A_k\to A$ in the $C^0$ topology. Moreover, $\lambda_{+}(\hat{A})=\lambda_{+}(A)$.

If $\widetilde{m}^k$ and $\widetilde{m}$ are defined by Equation (\ref{tutia2}), we infer that $m^k=(P^u\times Id)_*\widetilde{m}^k$ and $m=(P^u\times Id)_*\widetilde{m}$ are measures induced by $u$-states as in Definition \ref{induced}. Moreover, there exist continuous disintegrations $\{m^k_x\}$ and $\{m_x\}$ defined for every $x\in \Sigma^u$ such that $m^k_x$ are non-atomic measures for every $x\in \Sigma^u$ and $k\geq k_0$ and $m^k_x\to m_x$ uniformly on $x$. 

At this point, the argument follows in the same way as in \cite{BBB}, more precisely the results in Section 7 of \cite{BBB} conclude the proof of Theorem A. We provide a brief introduction of the method they use but we do not repeat the proof here. 

\subsubsection*{Part II: The energy method.}

The energy method was first introduced by Avila, Eskin and Viana \cite{AEV}, as the starting point of on ongoing project dealing with the continuity of Lyapunov exponents for random product of matrices in dimension higher than 2. This argument allow them to provide an alternative proof for \cite{BV}. We refer the reader to Chapter 10 of \cite{LLE} for a detailed explanation in that setting. Roughly speaking, they use the tools of coupling and energies to prove that the expanding point of $\mathbb{P}(A)$ defined by the stable subspace associated to $\lambda_{-}(A)$ is invisible for $\eta$ if $\eta$ is the limit measure of a sequence of non-atomic stationary measures $\eta_k$. 

In the context of \cite{BBB} the authors have to deal with a more general situation since there exists a non-atomic measure $m^k_x$ in $\mathbb{P}^1$ for every $x\in \Sigma^u$. However, it is possible to extend the energy method using Proposition \ref{uniform}: $m_x^k\to m_x$ uniformly. Therefore, they consider a suitable family of sets $U_x$ and apply the argument to the measures $\{m^k_x\vert_{U_x}\}_{x\in\Sigma}$. Their approach is closer to the higher dimensional version in \cite{AEV}, since they consider additive Margulis functions to conclude their result.

A special case of the energy method, using an explicit estimation, was used by Tall and Viana \cite{TV} to study the moduli of continuity of Lyapunov exponents of random products of matrices in dimension 2. See also Appendix A of \cite{SV}. $\square$

\section{Proof of Theorem B} 

Recall that $\hat{M}=\{1,...,d\}^{\mathbb{Z}}$ and $\hat{\mu}$ is a fully supported Bernoulli measure.

\begin{theoB} Let $\hat{A}\in \mathcal{S}_1(\hat{M},2)$ be a non-uniformly fiber-bunched, locally constant and irreducible cocycle. If $\hat{A}_k\to \hat{A}$ in $\mathcal{S}_1(\hat{M},2)$, then $\lambda_{+}(\hat{A}_k)\to \lambda_{+}(\hat{A})$. 
\end{theoB} 
\textit{Proof.}

Since $\hat{A}$ is a locally constant cocycle, then there exists a function $$A\colon \{1,...,d\}\to SL(2,\mathbb{R})$$ such that $\hat{A}(\hat{x})=A(x_0)$. Observe that this implies that $\hat{A}$ admits uniform invariant holonomies and both of them are the identity. 

Suppose that $\hat{A}$ is a dis\-con\-ti\-nu\-i\-ty point for the Lyapunov exponents, then $$\lambda_{+}(\hat{A})>0>\lambda_{-}(\hat{A})$$ and we can define measures $\hat{m}^s$ and $\hat{m}^u$ as in Equation (\ref{msu}).

Proposition \ref{discont} states that every $\mathbb{P}(\hat{A})$-invariant measure is an $su$-state, in particular, $\hat{m}^u$ is an $su$-state. Since $\hat{A}$ admits uniform invariant holonomies, Proposition \ref{yanise} gives us that $\xi:\hat{\Sigma}\to\mathbb{P}^1$, defined as $\xi(\hat{x})=\text{supp }\hat{m}^u_{\hat{x}}$ is a continuous section invariant by $\mathbb{P}(\hat{A})$ and both holonomies.

Since the uniform invariant holonomies are both the identity, we deduce that $\xi(\hat{x})=E$ where $E\in\mathbb{P}^1$ for every $\hat{x}\in\Sig$. Here is where we use that $\hat{M}$ is a full shift. 

Finally, as $\hat{m}^u$ is a $\mathbb{P}(\hat{A})$-invariant measure, we infer that for $\hat{x}$ almost every point $$\mathbb{P}(\hat{A}(\hat{x}))_*\hat{m}^u_{\hat{x}}=\hat{m}^u_{\hat{f}(\hat{x})}.$$  Moreover, since $\xi$ is continuous, we deduce that $\mathbb{P}(\hat{A}(\hat{x}))E=E$ for every $\hat{x}\in \hat{\Sigma}$. However, if the projectivization of $\hat{A}$ has a fixed point, then $\hat{A}$ has an invariant subspace, which contradicts the hypothesis of irreducibility and concludes the proof of the Theorem B. $\square$ 

\subsection*{Acknowledgments.} The first author would like to thank Marcelo Viana for the guidance and encouragement during her Ph.D. Thesis at IMPA which originates this work. The authors are grateful with Lorenzo J. D\'iaz and Lucas Backes for the useful observations and several corrections to the text. C. F. has been partially supported by IMPA and CAPES.

\end{document}